\documentclass[draft]{birkjour}

\usepackage[noadjust]{cite}
\usepackage{xcolor}
\RequirePackage[all]{xy}

\usepackage{longtable}
\usepackage{afterpage}

\usepackage{lscape}
\usepackage{hyperref}

\newtheorem{thm}{Theorem}[section]

\newtheorem{lem}[thm]{Lemma}

\theoremstyle{definition}

\theoremstyle{remark}
\newtheorem{rem}[thm]{Remark}

\newtheorem*{ex}{Example}
\numberwithin{equation}{section}

\newcommand{\BibTeX}{B\kern-0.1emi\kern-0.017emb\kern-0.15em\TeX}
\newcommand{\XYpic}{$\mathrm{X\kern-0.3em\raisebox{-0.18em}{Y}}$-$\mathrm{pic}\,$}

\newcommand{\cl}{C \kern -0.1em \ell}  

\newcommand{\BF}{\mathbb{F}}

\newcommand{\BR}{\mathbb{R}}
\newcommand{\BC}{\mathbb{C}}

\newcommand{\ed}{\end{document}}

\def\cl{{C}\!\ell}

\def\F{{\mathbb F}}

\def\BC{{\mathbb C}}

\def\C{\mathcal {G}}

\def\P{{\rm P}}
\def\A{{\rm A}}
\def\B{{\rm B}}
\def\Q{{\rm Q}}
\def\Z{{\rm Z}}
\def\H{{\rm H}}

\def\Z{{\rm Z}}
\def\Aut{{\rm Aut}}
\def\ker{{\rm ker}}

\def\Pin{{\rm Pin}}
\def\Spin{{\rm Spin}}
\def\ad{{\rm ad}}
\def\mod{{\rm \;mod\; }}
\def\Lambd{{\rm\Lambda}}
\def\Gamm{{\rm\Gamma}}
\def\Mat{{\rm Mat}}
\def\GL{{\rm GL}}
\def\UT{{\rm UT}}
\def\SUT{{\rm SUT}}
\def\Heis{{\rm Heis}}
\def\tr{{\rm tr}}

\begin{document}

%
%
%
%
%
%
%
%
%

\title[On Lie Groups Preserving Subspaces of Degenerate Clifford Algebras]
 {On Lie Groups Preserving Subspaces of Degenerate Clifford Algebras}
\author[E. Filimoshina]{Ekaterina Filimoshina}
\address{%
HSE University\\
Moscow 101000\\
Russia
\medskip}
\email{filimoshinaek@gmail.com}
%
\author[D. Shirokov]{Dmitry Shirokov}
\address{%
HSE University\\
Moscow 101000\\
Russia
\medskip}
\address{
and
\medskip}
\address{
Institute for Information Transmission Problems of the Russian Academy of Sciences \\
Moscow 127051 \\
Russia}
\email{dm.shirokov@gmail.com}
\subjclass{15A66, 11E88}
\keywords{geometric algebra, Clifford algebra, degenerate geometric algebra, Lipschitz group, Clifford group, spin group, grade involution, reversion, adjoint representation, twisted adjoint representation, Heisenberg group}
%
\begin{abstract}
This paper introduces Lie groups in degenerate geometric (Clifford) algebras that preserve four fundamental subspaces determined by the grade involution and reversion under the adjoint and twisted adjoint representations. We prove that these Lie groups can be equivalently defined using norm functions of multivectors applied in the theory of spin groups. We also study the corresponding Lie algebras. Some of these Lie groups and algebras are closely related to Heisenberg Lie groups and algebras. The introduced groups are interesting for various applications in physics and computer science, in particular, for constructing equivariant neural networks. 
\end{abstract}
\label{page:firstblob}
\maketitle

\section{Introduction}

In this paper, we consider degenerate and non-degenerate real and complex geometric (Clifford) algebras $\C_{p,q,r}$  of arbitrary dimension and signature (in the case of complex geometric algebras, we can take $q=0$). Degenerate geometric algebras have  applications in physics \cite{br1,br2}, geometry \cite{pga1,gunn1,ma}, computer vision and image processing \cite{bayro_new1,dorst_new,hild_new}, motion capture and robotics \cite{bayro_new1,ro2_new,sommer_new}, neural networks and machine learning \cite{tf,cNN0,cNN}, etc.

One of the most significant concepts in the theory of spin groups \cite{lg1,lounesto,p} is the twisted adjoint representation $\check{\ad}$. 
It is used to describe two-sheeted coverings of orthogonal groups by spin groups.
The representation $\check{\ad}$ was first introduced by Atiyah, Bott, and Shapiro in \cite{ABS} for a particular case needed in the context of spin groups.
There are two main approaches to extending this definition to the general case (see Section \ref{section_au}). 
In our paper, we give the answers to all the posed questions for both of them and denote the two versions of the twisted adjoint representation by $\check{\ad}$ and $\tilde{\ad}$ respectively. 
We consider the adjoint representation as well and denote it by $\ad$.
The well-known (see, for example, \cite{ABS,lg1,lounesto}) Clifford groups $\Gamma_{p,q,r}$ (\ref{def_cg}) and Lipschitz groups $\Gamma^{\pm}_{p,q,r}$ (\ref{def_lg}) preserve the grade-$1$ subspace  $\C^{1}_{p,q,r}$ (of vectors) under $\ad$ and $\check{\ad}$ respectively. 
The spin groups can be considered as normalized subgroups of the Lipschitz groups (see \cite{lg1,lounesto,p} and examples in the case $\C_{p,q,0}$ in Section \ref{section_centralizers}). 
They are defined using the norm functions $\psi(T):=\widetilde{T}T$ and $\chi(T):=\widehat{\widetilde{T}}T$ of the Lipschitz groups' elements $T\in\Gamma^{\pm}_{p,q,r}$.

In this work, we consider the analogues of the Clifford groups $\Gamma_{p,q,r}$ and Lipschitz groups $\Gamma^{\pm}_{p,q,r}$, which we call generalized Clifford and Lipschitz groups. 
These groups preserve the subspaces $\C^{\overline{k}}_{p,q,r}$, $k=0,1,2,3$, determined by the grade involution and reversion under $\ad$, $\check{\ad}$, and $\tilde{\ad}$.
By analogy with the notation of the Clifford and Lipschitz groups, we denote these groups by $\Gamma^{\overline{k}}_{p,q,r}$, $\check{\Gamma}^{\overline{k}}_{p,q,r}$, and $\tilde{\Gamma}^{\overline{k}}_{p,q,r}$ respectively. The mark above the letter indicates under which representation the subspace is preserved.
We introduce and study the families of Lie groups $\Q^{\overline{k}}_{p,q,r}$ and $\check{\Q}^{\overline{k}}_{p,q,r}$, $k=0,1,2,3$, that are defined by conditions imposed on the values of the norm functions  $\psi$ and $\chi$ of groups' elements. 
We prove that the groups $\Gamma^{\overline{k}}_{p,q,r}$, $\check{\Gamma}^{\overline{k}}_{p,q,r}$, and $\tilde{\Gamma}^{\overline{k}}_{p,q,r}$ coincide in different cases with the groups $\Q^{\overline{k}}_{p,q,r}$ and $\check{\Q}^{\overline{k}}_{p,q,r}$ (see Theorems \ref{maintheo_q} and \ref{maintheo_checkq}).
In these theorems and definitions of $\Q^{\overline{k}}_{p,q,r}$ and $\check{\Q}^{\overline{k}}_{p,q,r}$, we use centralizers and twisted centralizers of the subspaces of fixed grades $\C^{m}_{p,q,r}$  and the subspaces $\C^{\overline{k}}_{p,q,r}$ in $\C_{p,q,r}$ \cite{centr}. 
Theorems  \ref{maintheo_q}, \ref{maintheo_checkq}, \ref{thm_alg} and Lemmas  \ref{lemma_QABP}, \ref{lemma_chQ_eq} are new.

The current work generalizes the results of the papers \cite{OnInner} and \cite{GenSpin}, which focus on the non-degenerate geometric algebras $\C_{p,q,0}$ and consider groups preserving $\C^{\overline{k}}_{p,q,0}$ under $\ad$ and $\check{\ad}$ respectively.
In the previous works \cite{ICACGA,OnSomeLie}, we consider the groups preserving the even and odd subspaces under $\ad$, $\check{\ad}$, and $\tilde{\ad}$ in the case of any degenerate and non-degenerate $\C_{p,q,r}$ (see (\ref{prev_gr_0})--(\ref{p_f}) and the notes above). 
The papers \cite{AB_cgi,AB_aaca} consider the groups preserving other subspaces under $\ad$, $\check{\ad}$, and $\tilde{\ad}$ in the case of any $\C_{p,q,r}$.
In this work, we show the connection between these groups and the groups $\Gamma^{\overline{k}}_{p,q,r}$, $\check{\Gamma}^{\overline{k}}_{p,q,r}$, and $\tilde{\Gamma}^{\overline{k}}_{p,q,r}$. The current paper completes the study on the classification of Lie groups preserving fixed subspaces of geometric algebras under the adjoint and twisted adjoint representations, begun in the previous works.

In the case of small dimensions, the generalized degenerate Clifford and Lipschitz groups introduced in this paper are related to the well-known matrix groups (see Section~\ref{section_examples} for examples). 
In the special case of the Grassmann algebras $\C_{0,0,1}$ and $\C_{0,0,2}$, all the  considered groups  can be realized as subgroups of the upper triangular matrix groups $\UT(2,\F)$ and $\UT(4,\F)$. 
In the case of $\C_{0,0,2}$, the groups are closely related to the higher-dimensional Heisenberg group $\Heis_4$, which is applied in quantum mechanics and computing (see, for example, \cite{baker,heis2}).
Namely, the group $\Q^{\overline{0}}_{0,0,2}=\Q^{\overline{2}}_{0,0,2}=\check{\Q}^{\overline{1}}_{0,0,2}=\check{\Q}^{\overline{2}}_{0,0,2}$ is isomorphic to a matrix group, and its unipotent subgroup is a subgroup of $\Heis_4$.
We provide examples of the connections between the Lie algebras corresponding to the considered Lie groups and classical matrix and other Lie algebras as well.
For instance, in the case of $\C_{1,0,2}$, $\C_{1,1,1}$, and $\C_{3,0,0}$, some of the Lie algebras are isomorphic to the direct sums of $\mathfrak{su}(2)$, $\mathfrak{sl}(2,\BC)$, Heisenberg Lie algebras $\mathfrak{heis}(3,\F)$ and  $\mathfrak{heis}(4,\F)$, and Poincaré Lie algebra $\mathfrak{p}(1,1)$ (see examples in Section \ref{section_liealg}).

The generalized degenerate Clifford and Lipschitz groups can be useful for applications, for example, in neural networks, computer vision, and quantum mechanics.
The ordinary degenerate  Lipschitz groups  and Clifford groups \cite{Abl,br1,br2,cr,tu} have applications in construction of Clifford group equivariant neural networks \cite{cNN} and geometric Clifford algebra networks \cite{cNN0}, representation theory of Galilei group  in quantum mechanics \cite{br2}, rotor-based edge detection \cite{ce}, etc.
The generalized Lipschitz groups can be applied in construction of generalized Lipschitz group equivariant neural networks \cite{glgenn0,glgenn}.

This paper is organized as follows. Section \ref{section_Cpqr} introduces all the necessary notation related to the degenerate and non-degenerate geometric algebras $\C_{p,q,r}$.
Section \ref{section_au} recalls the notion of inner and twisted inner automorphisms and adjoint and twisted adjoint representations in $\C_{p,q,r}$.
In Section \ref{section_centralizers}, we recall the notion of centralizers and twisted centralizers and connect them with the groups considered in the next sections. 
Section \ref{section_relworks} discusses related works on groups preserving fixed subspaces of geometric algebras under the adjoint and twisted adjoint representations. 
In Section \ref{section_qt}, we consider the groups $\Gamma^{\overline{k}}_{p,q,r}$, $\check{\Gamma}^{\overline{k}}_{p,q,r}$, and $\tilde{\Gamma}^{\overline{k}}_{p,q,r}$, $k=0,1,2,3$, preserving the subspaces determined by the grade involution and reversion under $\ad$, $\check{\ad}$, and $\tilde{\ad}$ respectively and prove that they coincide with the groups $\Q^{\overline{k}}_{p,q,r}$ and $\check{\Q}^{\overline{k}}_{p,q,r}$.
Section \ref{section_examples} provides several examples of all the groups considered in the work. 
The corresponding Lie algebras are studied in Section~\ref{section_liealg}.
Conclusions follow in Section \ref{section_conclusions}.

\section{Geometric Algebras $\C_{p,q,r}$}\label{section_Cpqr}

Let us consider the geometric (Clifford) algebra \cite{HelmBook,hestenes,lounesto,p} $\C(V)=\C_{p,q,r}$, $p+q+r=n\geq1$, over a vector space $V$ with a symmetric bilinear form, where $V$ can be real $\BR^{p,q,r}$ or complex  $\BC^{p+q,0,r}$. We use $\F$ to denote the field of real numbers $\BR$ in the first case and the field of complex numbers $\BC$ in the second case. In this work, we consider both the case of the non-degenerate geometric algebras $\C_{p,q,0}$ and the case of the degenerate geometric algebras $\C_{p,q,r}$, $r\neq0$.
We use $\Lambda_r$ to denote the subalgebra $\C_{0,0,r}$, which is the Grassmann (exterior) algebra \cite{phys,lounesto}.
We use $\C:=\C_{p,q,r}$, $\C_{p,q}:=\C_{p,q,0}$, and $\Lambda:=\Lambda_r$ for brevity when considering arbitrary $p,q,r$, the non-degenerate geometric algebra case, and the Grassmann algebra case respectively, specifying the indices where needed.
The identity element is denoted by $e$, the generators are denoted by $e_a$, $a=1,\ldots,n$. The generators satisfy the conditions:
\begin{eqnarray}
    e_a e_b + e_b e_a = 2 \eta_{ab}e,\qquad \forall a,b=1,\ldots,n,
\end{eqnarray}
where $\eta=(\eta_{ab})$ is the diagonal matrix with $p$ times $+1$, $q$ times $-1$, and $r$ times $0$ on the diagonal in the real case $\C(\BR^{p,q,r})$ and $p+q$ times $+1$ and $r$ times $0$ on the diagonal in the complex case $\C(\BC^{p+q,0,r})$.

Consider the subspaces $\C^k$ of fixed grades $k=0,\ldots,n$. Their elements are linear combinations of the basis elements $e_A=e_{a_1\ldots a_k}:=e_{a_1}\cdots e_{a_k}$, $a_1<\cdots<a_k$, labeled by ordered multi-indices $A$ of length $k$, where $0\leq k\leq n$. The multi-index with zero length $k=0$ corresponds to the identity element $e$. 
The grade-$0$ subspace is always denoted by $\C^0$ without the lower indices $p,q,r$, since it does not depend on the geometric algebra's signature.
We have $\C^{k}=\{0\}$ for $k<0$ and $k>n$.
We use the following notation:
\begin{eqnarray}\label{def_Cgeqk}
\C^{\geq k}:=\C^{k}\oplus\C^{k+1}\oplus\cdots\oplus\C^{n},\;\; \C^{\leq k}:=\C^0\oplus\C^{1}\oplus\cdots\oplus\C^{k}
\end{eqnarray}
for $0\leq k\leq n$. For example, $\C^{\geq0}=\C^{\leq n}=\C$ and $\C^{\geq n}=\C^{n}$. 

The spaces $\C^k_{p,q}$ and $\Lambda^k_r$, $k=0,...,n$, can be regarded as subspaces of $\C_{p,q,r}$. By $\{\C^k_{p,q}\Lambda^{l}_r\}$, we denote
 the subspace of $\C_{p,q,r}$ spanned by the elements of the form $ab$, where $a\in\C^k_{p,q}$ and $b\in\Lambda^l_r$.

In this work, we use such conjugation operations as grade involution, reversion, and their composition, which is usually called Clifford conjugation. 
The grade involute of an element  $U\in\C$ is denoted by $\widehat{U}$, the reverse is denoted 
by $\widetilde{U}$, the Clifford conjugate is denoted by $\widehat{\widetilde{U}}$.

The grade involution defines the even $\C^{(0)}$ and odd $\C^{(1)}$ subspaces:
\begin{eqnarray}\label{def_even_odd}
\C^{(k)} = \{U\in\C:\quad \widehat{U}=(-1)^kU\}=\bigoplus_{j=k\mod{2}}\C^j,\qquad k=0,1.
\end{eqnarray}
We can represent any element $U\in\C$ as a sum $U=U_{(0)}+U_{(1)}$, where $U_{(0)}\in\C^{(0)}$ and $U_{(1)}\in\C^{(1)}$. We use the operations of projection 
 $\langle\rangle_{(0)}$ and $\langle\rangle_{(1)}$ onto the even $\C^{(0)}$ and odd $\C^{(1)}$ subspaces respectively:
\begin{eqnarray}\label{proj1}
    \langle U \rangle_{(0)} := U_{(0)},\qquad \langle U \rangle_{(1)} := U_{(1)}.
\end{eqnarray}
For an arbitrary subset $\H\subseteq\C$, we have
\begin{eqnarray}
    \langle \H \rangle_{(0)} := \H\cap\C^{(0)},\qquad  \langle \H \rangle_{(1)} := \H\cap\C^{(1)}.\label{H_even}
\end{eqnarray}

The grade involution and reversion define four subspaces $\C^{\overline{0}}$,  $\C^{\overline{1}}$,  $\C^{\overline{2}}$, and  $\C^{\overline{3}}$ (they are called the subspaces of quaternion types $0, 1, 2$, and $3$ respectively in the papers \cite{quat1, quat2, quat3}):
\begin{eqnarray}
\C^{\overline{k}}=\{U\in\C:\;\; \widehat{U}=(-1)^k U,\;\; \widetilde{U}=(-1)^{\frac{k(k-1)}{2}} U\},\;\; k=0, 1, 2, 3.\label{qtdef}
\end{eqnarray}
The geometric algebra $\C$ can be represented as a direct sum $\C = \C^{\overline{0}}\oplus\C^{\overline{1}}\oplus\C^{\overline{2}}\oplus\C^{\overline{3}}$ and viewed as $\mathbb{Z}_2\times\mathbb{Z}_2$-graded algebra with respect to the commutator and anticommutator \cite{b_lect}.
To denote the direct sum of different subspaces, we use the upper multi-index and omit the direct sum sign. For instance, $\C^{(1)\overline{2}4}:=\C^{(1)}\oplus\C^{\overline 2}\oplus\C^{4}$.

\section{Inner and Twisted Inner Automorphisms in Geometric Algebras}\label{section_au}

We use the upper index $\times$ to denote the subset $\H^{\times}$ of all invertible elements of any set $\H$. Consider the well-known (see, for example, \cite{ABS,lg1,lounesto}) Clifford and Lipschitz groups, which are defined in the following way:
\begin{eqnarray}
    \Gamma&:=&\Gamma_{p,q,r}:=\{T\in\C^{\times}:\quad T\C^{1}T^{-1}\subseteq\C^{1}\},\label{def_cg}
    \\
    \Gamma^{\pm}&:=&\Gamma^{\pm}_{p,q,r}:=\{T\in\C^{\times}:\quad \widehat{T}\C^{1}T^{-1}\subseteq\C^{1}\}.\label{def_lg}
\end{eqnarray}

Let us consider the adjoint representation $\ad$ acting on the group of all invertible elements $\ad:\C^{\times}\rightarrow\Aut(\C)$ as $T\mapsto\ad_T$, where $\ad_{T}:\C\rightarrow\C$:
\begin{eqnarray}\label{ar}
\ad_{T}(U):=TU T^{-1},\qquad U\in\C,\qquad T\in\C^{\times}.
\end{eqnarray}
Also let us consider the twisted adjoint representation, which is introduced in a particular case in the paper \cite{ABS}. 
There are two possible general definitions of the twisted adjoint representation in the literature, we use the notation $\check{\ad}$ and $\tilde{\ad}$ for them\footnote{We also want to note that an anonymous reviewer wisely suggested calling $\tilde{\ad}$ a  {\it super adjoint representation} as this is the natural action of the group of units
of a superalgebra on the superalgebra.}.
These operations are acting $\check{\ad},\tilde{\ad}:\C^{\times}\rightarrow\Aut(\C)$ as $T\mapsto\check{\ad}_T$, $T\mapsto\tilde{\ad}_T$ with $\check{\ad}_{T},\tilde{\ad}_{T}:\C\rightarrow\C$:
\begin{eqnarray}
\!\!\!\!\!\!\!\!\!\!\!\!\!\!\!&&\check{\ad}_{T}(U):=\widehat{T}U T^{-1},\quad U\in\C,\quad T\in\C^{\times},\label{twa1}
\\
\!\!\!\!\!\!\!\!\!\!\!\!\!\!\!&&\tilde{\ad}_{T}(U):=T\langle U\rangle_{(0)} T^{-1}+\widehat{T} \langle U\rangle_{(1)} T^{-1},\quad \forall U\in\C,\quad T\in\C^{\times}.\label{twa22}
\end{eqnarray}
For elements of fixed parity, we have 
\begin{eqnarray}
&&\tilde{\ad}_{T}(U_{(0)})=\ad_T(U_{(0)}),\qquad \forall U_{(0)}\in\C^{(0)},\qquad T\in\C^{\times},\label{ad_t_1}\\ 
&&\tilde{\ad}_{T}(U_{(1)})=\check{\ad}_T(U_{(1)}),\qquad \forall U_{(1)}\in\C^{(1)},\qquad T\in\C^{\times}.\label{ad_t_2}
\end{eqnarray}
Answers to all the questions for $\check{\ad}$ and $\tilde{\ad}$ posed in this work  are closely related,  and we consider both these operations for the readers' convenience. 

Kernels of $\ad$, $\check{\ad}$, and $\tilde{\ad}$ have the form (see, for example, \cite{OnSomeLie}):
\begin{eqnarray}
    \!\!\!\!\!\!\!\!\!\!&&\!\!\!\!\!\!\!\!\!\!\!\!\!\!\!\!\!\ker{(\ad)}=\{T\in\C^{\times}:\; T U T^{-1}=U,\; \forall U \in\C\}=\Z^{\times},\label{ker_ad}
    \\
   \!\!\!\!\!\!\!\!\!\!&&\!\!\!\!\!\!\!\!\!\!\!\!\!\!\!\!\!\ker(\check{\ad})=\{T\in\C^{\times}:\; \widehat{T}UT^{-1}=U,\; \forall U\in\C\}=\Lambda^{(0)\times},\label{ker_check_ad}
    \\
    \!\!\!\!\!\!\!\!\!\!&&\!\!\!\!\!\!\!\!\!\!\!\!\!\!\!\!\!\ker(\tilde{\ad})=\{T\in\C^{\times}:\; T \langle U\rangle_{(0)} T^{-1} \!\!+\!\widehat{T}\langle U\rangle_{(1)}T^{-1}\!=\!U,\; \forall U\in\C\}\!=\!\Lambda^{\times}.\label{ker_tilde_ad}
\end{eqnarray}

\section{Preliminaries on Centralizers and Twisted Centralizers}\label{section_centralizers}

In this section, we consider the notion of centralizers and twisted centralizers of the subspaces of fixed grades $\C^m$ and the subspaces $\C^{\overline{k}}$ determined by the grade involution and reversion in $\C$. These centralizers and twisted centralizers are considered in detail in the paper \cite{centr}. In the current work, we use some of the centralizers and twisted centralizers in definitions of the groups in Section \ref{section_qt} and need their explicit forms to prove the main statements.

Let us consider centralizers $\Z^{m}$ and twisted centralizers $\check{\Z}^{m}$ of the subspaces $\C^{m}$ of fixed grades, $m=0,1,\ldots,n$, in $\C$:
\begin{eqnarray}
\Z^{m}&:=&\{X\in\C:\quad X V = V X,\quad \forall V\in\C^{m}\},\label{def_Zm}
\\
    \check{\Z}^m&:=&\{X\in\C:\quad \widehat{X} V = V X,\quad \forall V\in\C^{m}\}.\label{def_chZm}
\end{eqnarray}
The center of $\C$ is also the centralizer; we denote it by $\Z:=\Z^1$. It is well known (see, for example, \cite{Abl,br1}) that
\begin{eqnarray}\label{Zpqr}
\Z=
\left\lbrace
\begin{array}{lll}
\Lambd^{(0)}\oplus\C^n,&&\mbox{$n$ is odd},
\\
\Lambd^{(0)},&&\mbox{$n$ is even}.
\end{array}
\right.
\end{eqnarray}
We have ${\Z}^{m}=\check{{\Z}}^{m}=\C$ for $m<0$ and $m>n$.
Explicit forms of $\Z^m$ and $\check{\Z}^m$, $m=0,\ldots,n$, are provided in \cite{centr}.
For the sets $\Z^m$ and $\check{\Z}^m$, $m=0,\ldots,n$, the projections (\ref{H_even}) onto the even and odd subspaces $\C^{(l)}$, $l=0,1$, have the form 
\begin{eqnarray}
\langle\Z^m\rangle_{(l)}:= \Z^m\cap\C^{(l)}=\{\langle X\rangle_{(l)}\;|\;X\in\Z^m\},
\\
\langle\check{\Z}^m\rangle_{(l)}:= \check{\Z}^m\cap\C^{(l)}=\{\langle X\rangle_{(l)}\;|\;X\in\check{\Z}^m\},
\end{eqnarray}
since $\Z^m$ and $\check{\Z}^m$ are direct sums of fixed subspaces of the special form \cite{centr}.
The projections $\langle\Z^m\rangle_{(0)}$ and $\langle\check{\Z}^m\rangle_{(0)}$ coincide by definition:
\begin{eqnarray}
    \langle\Z^m\rangle_{(0)}=\langle\check{\Z}^m\rangle_{(0)},\qquad \forall m=0,1,\ldots,n.\label{same_even}
\end{eqnarray}

\begin{rem}\label{CC3CC4}
    Further, we use the following facts implied by Theorem 3.6 \cite{centr} about explicit forms of centralizers and twisted centralizers:
\begin{eqnarray}
\!\!\!\!\!\!\!\!\!\!\!\!\!\!&\ker(\ad)=\Z^{1\times}\subseteq\Z^{m\times},\quad \ker(\check{\ad})=\Lambda^{(0)\times}\subseteq\check{\Z}^{m\times},\quad  m=0,\ldots,n;\label{CC3CC4_2}
\\
\!\!\!\!\!\!\!\!&\Z^{m}\subseteq\Z^{4},\qquad \check{\Z}^{m}\subseteq\Z^{4},\qquad m=1,2,3.\label{CC3CC4_3}
\end{eqnarray}
\end{rem}

Let us also consider the centralizers $\Z^{\overline{m}}$ and twisted centralizers $\check{\Z}^{\overline{m}}$ of the subspaces $\C^{\overline{m}}$ (\ref{qtdef}), $m=0,1,2,3$, in $\C$:
\begin{eqnarray}
\!\!\!\!\!\!\!\!\!\!\!\!\!\!\!\Z^{\overline{m}}\!\!\!&:=&\!\!\!\{X\in\C:\quad X V = V X,\quad \forall V\in\C^{\overline{m}}\},\quad m=0,1,2,3,\label{def_CC_ov}
\\
\!\!\!\!\!\!\!\!\!\!\!\!\!\!\!\check{\Z}^{\overline{m}}\!\!\!&:=&\!\!\!\{X\in\C:\quad \widehat{X} V = V X,\quad \forall V\in\C^{\overline{m}}\},\quad m=0,1,2,3.\label{def_chCC_ov}
\end{eqnarray}
These sets are considered in detail in the paper \cite{centr}. Further, we use the following lemma  about their explicit forms.
\begin{lem}[Theorem 5.1 \cite{centr}]\label{centralizers_qt}
We have
\begin{eqnarray*}
&\Z^{\overline{m}}=\Z^{m},\quad\check{\Z}^{\overline{m}}=\check{\Z}^{m},\quad m=1,2,3;
\\
&\Z^{\overline{0}}=\Z^{4},\quad \check{\Z}^{\overline{0}}=\langle\Z^{4}\rangle_{(0)}.
\end{eqnarray*}
\end{lem}

Let us consider two norm functions, which are widely used in the theory of spin groups \cite{ABS,lg1,lounesto}:
\begin{eqnarray}\label{norm_functions}
    \psi(T):=\widetilde{T}T,\qquad \chi(T):=\widehat{\widetilde{T}}T,\qquad \forall T\in\C.
\end{eqnarray}
For example, in the case of the non-degenerate geometric algebra $\C_{p,q}$, the groups  $\Pin_{p,q}$ and $\Spin_{p,q}$ are defined as  (\cite{lounesto,p}):
\begin{eqnarray}
    \!\!\!\!\!\!\!\!\!\!\!\!\!\!\Pin_{p,q}\!\!\!\!\!&:=&\!\!\!\!\!\{T\in \Gamma^{\pm}_{p,q}:\; \psi(T)=\pm e\}\!=\!\{T\in \Gamma^{\pm}_{p,q}:\; \chi(T)=\pm e\},\label{pin_ex}
    \\
    \!\!\!\!\!\!\!\!\!\!\!\!\!\!\Spin_{p,q}\!\!\!\!\!&:=&\!\!\!\!\!\{T\in \langle\Gamma^{\pm}_{p,q}\rangle_{(0)}:\psi(T)=\pm e\}\!=\!\{T\in \langle\Gamma^{\pm}_{p,q}\rangle_{(0)}:\chi(T)=\pm e\}.\label{spin_ex}
\end{eqnarray}
Consider Lemma \ref{lemma_for_AB} about such $T\in\C$ that $\psi(T)$ and $\chi(T)$
are in the centralizers $\Z^{\overline{m}}$ (\ref{def_CC_ov}) and twisted centralizers $\check{\Z}^{\overline{m}}$ (\ref{def_chCC_ov}). 
The proofs of Theorems \ref{maintheo_q} and \ref{maintheo_checkq} are based on this lemma.

\begin{lem}[Lemma 3.3 \cite{AB_aaca}]\label{lemma_for_AB}
For any $T\in\C^{\times}$, in the cases $(k,l)=(0,1),(1,0),$ $(2,3),(3,2)$, we have:
\begin{eqnarray*}
 \!\!\!\!\!\!\!&T\C^{\overline{k}} T^{-1}\subseteq \C^{\overline{kl}} \Leftrightarrow\psi(T)\in\Z^{\overline{k}\times},\qquad \widehat{T} \C^{\overline{k}}T^{-1}\subseteq \C^{\overline{kl}}\Leftrightarrow\chi(T)\in\check{\Z}^{\overline{k}\times},
 \end{eqnarray*}
and in the cases $(k,l)=(0,3),(3,0),(1,2),(2,1)$, we have:
 \begin{eqnarray*}
  \!\!\!\!\!\!\!&T\C^{\overline{k}}T^{-1}\subseteq \C^{\overline{kl}}\Leftrightarrow\chi(T)\in\Z^{\overline{k}\times},\qquad\widehat{T} \C^{\overline{k}} T^{-1}\subseteq \C^{\overline{kl}}\Leftrightarrow\psi(T)\in\check{\Z}^{\overline{k}\times}.
\end{eqnarray*}
\end{lem}

\section{Related Works}\label{section_relworks}

In this work, we consider the mappings $\ad_T$ (\ref{ar}),  $\tilde{\ad}_T$ (\ref{twa22}), and $\check{\ad}_T$ (\ref{twa1}) that leave invariant the subspaces $\C^{\overline{k}}$ (\ref{qtdef}), $k=0,1,2,3$,  determined by the grade involution and reversion in degenerate and non-degenerate geometric algebras $\C$. This work generalizes the results of the papers \cite{OnInner} and \cite{GenSpin}, where the same questions regarding $\ad_T$ and $\check{\ad}_T$ respectively are considered in the case of the non-degenerate algebra $\C_{p,q}$. 

The paper \cite{OnSomeLie} discusses the groups preserving the subspaces of fixed parity $\C^{(0)}$ and $\C^{(1)}$ under the adjoint representation $\ad$ and the twisted adjoint representations $\check{\ad}$ and $\tilde{\ad}$ in the degenerate and non-degenerate $\C$. That paper introduces the four Lie groups
\begin{eqnarray}
    \P^{\pm}:=\P^{\pm}_{p,q,r}:=\C^{(0)\times}\cup\C^{(1)\times},&&
    \P := \P_{p,q,r}:=\P^{\pm}\Z^{\times},\label{prev_gr_0}
    \\
    \P^{\pm\Lambda} := \P^{\pm\Lambda}_{p,q,r}:=\P^{\pm}\Lambda^{\times},&&
    \P^{\Lambda} :=\P^{\Lambda}_{p,q,r}:= \P\Lambda^{\times},\label{prev_gr_0_0}
\end{eqnarray}
and proves that 
\begin{eqnarray}
&&\P=\Gamma^{(1)}\subseteq\P^{\Lambda}=\Gamma^{(0)}=\tilde{\Gamma}^{(0)},\label{prev_gr}
\\
&&\P^{\pm}=\check{\Gamma}^{(0)}\subseteq\P^{\pm\Lambda}=\check{\Gamma}^{(1)}=\tilde{\Gamma}^{(1)},\label{prev_gr_2}
\end{eqnarray}
where, for $k=0,1$,
\begin{eqnarray}
  \Gamma^{(k)}&:=&\Gamma^{(k)}_{p,q,r}:=\{T\in\C^{\times}:\; \ad_{T}(\C^{(k)})\subseteq\C^{(k)}\},\label{p_f1}
  \\
  \check{\Gamma}^{(k)}&:=&\check{\Gamma}^{(k)}_{p,q,r}:=\{T\in\C^{\times}:\; \check{\ad}_{T}(\C^{(k)})\subseteq\C^{(k)}\},\label{p_f0}
  \\
  \tilde{\Gamma}^{(k)}&:=&\tilde{\Gamma}^{(k)}_{p,q,r}:=\{T\in\C^{\times}:\; \tilde{\ad}_{T}(\C^{(k)})\subseteq\C^{(k)}\}.\label{p_f}
\end{eqnarray}

The papers \cite{AB_cgi,AB_aaca} introduce and study the groups preserving direct sums of the subspaces $\C^{\overline{k}}$, $k=0,1,2,3$, under $\ad$, $\check{\ad}$, and $\tilde{\ad}$ in any degenerate and non-degenerate $\C$. The following groups defined using centralizers \eqref{def_Zm}, twisted centralizers \eqref{def_chZm}, and norm functions \eqref{norm_functions}  are considered:
\begin{eqnarray}
\!\!\!\!\!\!\!\!\!\!\!\!\!\!\!\!\!\!\!\!&\A^{\overline{01}} := \A^{\overline{01}}_{p,q,r}:=\{T\in\C^{\times}: \quad {\widetilde{T}}T\in\Z^{1\times}\}=\psi^{-1}(\Z^{1\times}), \label{def_A01}
\\
\!\!\!\!\!\!\!\!\!\!\!\!\!\!\!\!\!\!\!\!&\B^{\overline{12}} := \B^{\overline{12}}_{p,q,r}:=\{T\in\C^{\times}:\quad \widehat{\widetilde{T}}T\in\Z^{1\times}\}=\chi^{-1}(\Z^{1\times}),\label{def_B12}
\\
\!\!\!\!\!\!\!\!\!\!\!\!\!\!\!\!\!\!\!\!&\A^{\overline{23}}:=\A^{\overline{23}}_{p,q,r}:=
 \{T\in\C^{\times}: \quad\widetilde{T}T\in(\Z^{2}\cap\Z^{3})^{\times}\}=\psi^{-1}((\Z^{2}\cap\Z^{3})^{\times}),\label{def_A23}
\\
\!\!\!\!\!\!\!\!\!\!\!\!\!\!\!\!\!\!\!\!& \B^{\overline{03}} := \B^{\overline{03}}_{p,q,r}:=\{T\in\C^{\times}: \quad\widehat{\widetilde{T}}T\in\Z^{3\times}\}=\chi^{-1}(\Z^{3\times}),\label{def_B03}
\\
\!\!\!\!\!\!\!\!\!\!\!\!\!\!\!\!\!\!\!\!&\check{\A}^{\overline{12}} :=\check{\A}^{\overline{12}}_{p,q,r}:=\{T\in\C^{\times}:\quad\widetilde{T}T\in(\check{\Z}^{1}\cap\check{\Z}^{2})^{\times}\}=\psi^{-1}((\check{\Z}^{1}\cap\check{\Z}^{2})^{\times}),\label{def_chA12}
\\
\!\!\!\!\!\!\!\!\!\!\!\!\!\!\!\!\!\!\!\!&\check{\A}^{\overline{03}} :=\check{\A}^{\overline{03}}_{p,q,r}:= \{T\in\C^{\times}:\quad\widetilde{T}T\in(\Z^{3}\cap\C^{(0)})^{\times}\}=\psi^{-1}((\Z^{3}\cap\C^{(0)})^{\times}),\label{def_chA03}
\\
\!\!\!\!\!\!\!\!\!\!\!\!\!\!\!\!\!\!\!\!&\check{\B}^{\overline{01}}:=\check{\B}^{\overline{01}}_{p,q,r}:= \{T\in\C^{\times}\!: \quad\widehat{\widetilde{T}}T\in(\Z^1\cap\C^{(0)})^{\times}\}=
\chi^{-1}((\Z^1\cap\C^{(0)})^{\times}),\label{def_chB01}
\\
\!\!\!\!\!\!\!\!\!\!\!\!\!\!\!\!\!\!\!\!& \check{\B}^{\overline{23}}:=\check{\B}^{\overline{23}}_{p,q,r}:=\{T\in\C^{\times}:\quad \widehat{\widetilde{T}}T\in(\check{\Z}^2\cap\check{\Z}^3)^{\times}\}=\chi^{-1}((\check{\Z}^2\cap\check{\Z}^3)^{\times}),\label{def_chB23}
\end{eqnarray}
and we have
\begin{eqnarray}
     \!\!\!\!\!\!\!\!\!\!&\A^{\overline{01}} = \Gamma^{\overline{01}},\quad \A^{\overline{23}} = \Gamma^{\overline{23}},\quad \B^{\overline{12}}= \Gamma^{\overline{12}},\quad \B^{\overline{03}}= \Gamma^{\overline{03}},\label{prev_gr_3}
    \\
     \!\!\!\!\! \!\!\!\!\!&\check{\A}^{\overline{12}}= \check{\Gamma}^{\overline{12}},\quad \check{\A}^{\overline{03}} = \check{\Gamma}^{\overline{03}},\quad
    \check{\B}^{\overline{01}}= \check{\Gamma}^{\overline{01}},\quad \check{\B}^{\overline{23}} = \check{\Gamma}^{\overline{23}},\label{prev_gr_4}
\end{eqnarray}
where, for $k,l=0,1,2,3$,
\begin{eqnarray}
    \Gamma^{\overline{kl}}&:=&\Gamma^{\overline{kl}}_{p,q,r}:= \{T\in\C^{\times}:\quad \ad_{T}(\C^{\overline{kl}})\subseteq\C^{\overline{kl}}\},\label{def_Gamma_kl}
    \\
    \check{\Gamma}^{\overline{kl}} &:=&\check{\Gamma}^{\overline{kl}}_{p,q,r}:= \{T\in\C^{\times}:\quad \check{\ad}_{T}(\C^{\overline{kl}})\subseteq\C^{\overline{kl}}\}. \label{def_chGamma_kl}
\end{eqnarray}

We show the connections between the groups (\ref{prev_gr})--(\ref{prev_gr_2}), \eqref{prev_gr_3}--\eqref{prev_gr_4} and the groups considered in Section \ref{section_qt} of this work.

\section{The Groups Preserving the Subspaces Determined by the Grade Involution and Reversion Under $\ad$, $\check{\ad}$, and $\tilde{\ad}$}\label{section_qt}

Let us consider setwise stabilizers (see, for example, \cite{alg}) of the  subspaces $\C^{\overline{k}}$ (\ref{qtdef}), $k=0,1,2,3$, in the group $\C^{\times}$ under the group actions $\ad$, $\check{\ad}$, and $\tilde{\ad}$.
We use the following notation for the groups preserving these subspaces under the adjoint representation $\ad$ (\ref{ar}):
\begin{eqnarray}
\Gamma^{\overline{k}} := \Gamma^{\overline{k}}_{p,q,r}:=\{T\in\C^{\times}:\quad \ad_T(\C^{\overline{k}}):=T\C^{\overline{k}}T^{-1}\subseteq\C^{\overline{k}}\}.\label{gamma_ov_k}
\end{eqnarray}
Note that the groups $\Gamma^{\overline{k}}$ are the normalizers of  $\C^{\overline{k}}$ in $\C^{\times}$.
The following notation is used for the stabilizers of $\C^{\overline{k}}$, $k=0,1,2,3$, in $\C^{\times}$ under the twisted adjoint representation $\check{\ad}$ (\ref{twa1}):
\begin{eqnarray}
\check{\Gamma}^{\overline{k}} :=\check{\Gamma}_{p,q,r}^{\overline{k}}:= \{T\in\C^{\times}:\quad \check{\ad}_T(\C^{\overline{k}}):=\widehat{T}\C^{\overline{k}}T^{-1}\subseteq\C^{\overline{k}}\},\label{gamma_ov_chk}
\end{eqnarray}
and $\tilde{\ad}$ (\ref{twa22}):
\begin{eqnarray}
\tilde{\Gamma}^{\overline{k}}:=\tilde{\Gamma}^{\overline{k}}_{p,q,r}:= \{T\in\C^{\times}:\quad \tilde{\ad}_T(\C^{\overline{k}})\subseteq\C^{\overline{k}}\}.\label{gamma_ov_tik}
\end{eqnarray}
The groups $\tilde{\Gamma}^{\overline{k}}$, $k=0,1,2,3$, are related to the groups ${\Gamma}^{\overline{k}}$ and $\check{\Gamma}^{\overline{k}}$ as follows:
\begin{eqnarray}\label{def_tilde_gk}
    \tilde{\Gamma}^{\overline{k}}=
    \left\lbrace
    \begin{array}{lll}
    \check{\Gamma}^{\overline{k}},&&k=1,3,
    \\
    {\Gamma}^{\overline{k}},&& k=0,2,
    \end{array}
    \right.
\end{eqnarray}
since $\tilde{\ad}_T(\C^{\overline{k}})={\ad}_T(\C^{\overline{k}})$ in the cases $k=0,2$ by (\ref{ad_t_1}) and $\tilde{\ad}_T(\C^{\overline{k}})=\check{\ad}_T(\C^{\overline{k}})$ in the cases $k=1,3$ by (\ref{ad_t_2}). In Sections \ref{section_Q} and \ref{section_chQ}, we find the equivalent definitions of the groups  $\Gamma^{\overline{k}}$ and $\check{\Gamma}^{\overline{k}}$, $k=0,1,2,3$.

\subsection{The Groups $\Q^{\overline{0}}$, $\Q^{\overline{1}}$, $\Q^{\overline{2}}$, $\Q^{\overline{3}}$, $\Gamma^{\overline{0}}$, $\Gamma^{\overline{1}}$, $\Gamma^{\overline{2}}$, and $\Gamma^{\overline{3}}$}\label{section_Q}

Let us consider the groups $\Q^{\overline{1}}$, $\Q^{\overline{2}}$, $\Q^{\overline{3}}$, and $\Q^{\overline{0}}$:
\begin{eqnarray}
\!\!\!\!\!\!\!\!\!\!\!\!\!\!\!\!\!\Q^{\overline{1}}\!\!\!\!\!&:=&\!\!\!\!\!\Q^{\overline{1}}_{p,q,r}:=\{ T\in\C^{\times}:\;\; \widetilde{T}T\in\Z^{1\times}=\ker(\ad),\;\;\widehat{\widetilde{T}}T\in\Z^{1\times}=\ker(\ad)\}\label{def_Q1}
\\
\!\!\!\!\!\!\!\!\!&=& \psi^{-1}(\Z^{1\times})\cap\chi^{-1}(\Z^{1\times}),
\\
\!\!\!\!\!\!\!\!\!\!\!\!\!\!\!\!\!\Q^{\overline{2}}\!\!\!\!\!&:=&\!\!\!\!\!\Q^{\overline{2}}_{p,q,r}:=\{ T\in\C^{\times}:\;\; \widetilde{T}T\in\Z^{2\times},\;\; \widehat{\widetilde{T}}T\in\Z^{2\times}\}\label{def_Q2}
\\
\!\!\!\!\!\!\!\!\!&=& \psi^{-1}(\Z^{2\times})\cap\chi^{-1}(\Z^{2\times}),
\\
\!\!\!\!\!\!\!\!\!\!\!\!\!\!\!\!\!\Q^{\overline{3}}\!\!\!\!\!&:=&\!\!\!\!\!\Q^{\overline{3}}_{p,q,r}:=\{T\in\C^{\times}:\;\; \widetilde{T}T\in\Z^{3\times},\;\; \widehat{\widetilde{T}}T\in\Z^{3\times}\}\label{def_Q3}
\\
\!\!\!\!\!\!\!\!\!&=& \psi^{-1}(\Z^{3\times})\cap\chi^{-1}(\Z^{3\times}),
\\
\!\!\!\!\!\!\!\!\!\!\!\!\!\!\!\!\!\Q^{\overline{0}}\!\!\!\!\!&:=&\!\!\!\!\!\Q^{\overline{0}}_{p,q,r}:=\{T\in\C^{\times}:\;\; \widetilde{T}T\in\Z^{4\times},\;\; \widehat{\widetilde{T}}T\in\Z^{4\times}\}\label{def_Q0}
\\
\!\!\!\!\!\!\!\!\!&=& \psi^{-1}(\Z^{4\times})\cap\chi^{-1}(\Z^{4\times}),
\end{eqnarray}
where $\ker(\ad)$ (\ref{ker_ad}) is the kernel of the adjoint representation $\ad$ (\ref{ar}), $\Z^1$, $\Z^2$, $\Z^3$, and $\Z^4$ are the centralizers of the subspaces $\C^1$, $\C^{2}$, $\C^{3}$, and $\C^{4}$ respectively considered in Section \ref{section_centralizers}. By \cite{centr}, we have
\begin{eqnarray*}
\!\!\!\!\!\!\!\!\!\!\!\!&&\Z^1=\Z=\left\lbrace
    \begin{array}{lll}
    \Lambda^{(0)}\oplus\C^{n},&&\mbox{$n$ is odd},
    \\
    \Lambda^{(0)},&&\mbox{$n$ is even},
     \end{array}
    \right.\label{cc_1}
\\
\!\!\!\!\!\!\!\!\!\!\!\!&&\Z^2=
\left\lbrace
    \begin{array}{lll}
\Lambda\oplus\C^{n},&& r\neq n,
\\
\Lambda,&& r=n,
\end{array}
    \right.\label{cc_2}
\\
\!\!\!\!\!\!\!\!\!\!\!\!&&\Z^3\!=\!\left\lbrace
            \begin{array}{lll}\label{cc_3}
            \!\!\!\Lambda^{(0)}\!\oplus\!\Lambda^{n-2}\oplus \{\C^{1}_{p,q}(\Lambda^{n-3}\oplus\Lambda^{n-2})\}\oplus \{\C^{2}_{p,q}\Lambda^{n-3}\}\oplus\C^{n} &&\mbox{$n$ is odd},
            \\
            \!\!\!\Lambda^{(0)}\!\oplus\!\Lambda^{n-1}\oplus \{\C^{1}_{p,q}\Lambda^{\geq n-2}\}\oplus \{\C^{2}_{p,q}\Lambda^{n-2}\}, &&\mbox{$n$ is even},
            \end{array}
            \right.\nonumber
\\
\!\!\!\!\!\!\!\!\!\!\!\!&&\Z^4=\Lambda\oplus \{\C^{1}_{p,q}(\Lambda^{n-3}\oplus\Lambda^{n-2})\}\oplus \{\C^{2}_{p,q}(\Lambda^{n-4}\oplus\Lambda^{n-3})\}\oplus\C^{n}.\label{cc_4}
\end{eqnarray*}
The cases of $\Z^{2}$ and $\check{\Z}^1$ are also proved in detail in \cite{OnSomeLie}.

The groups $\Q^{\overline{1}}$, $\Q^{\overline{2}}$, $\Q^{\overline{3}}$, and $\Q^{\overline{0}}$ are generalizations of the groups $\Q$ and $\Q'$ \cite{OnInner} in the non-degenerate geometric algebras $\C_{p,q}$  to the case of the degenerate geometric algebras $\C$ and coincide with them if $r=0$:
\begin{eqnarray*}
    \Q^{\overline{1}}_{p,q}=\Q^{\overline{3}}_{p,q}=\Q,\quad
     \Q^{\overline{0}}_{p,q}=\Q^{\overline{2}}_{p,q}=
     \left\lbrace
    \begin{array}{lll}
    \Q,\;\;n=1,2,3\mod{4},
    \\
    \Q',\;\;n=0\mod{4},
    \end{array}
    \right.
    \; n\neq2,3,
\end{eqnarray*}
and $\Q^{\overline{1}}_{p,q}=\Q^{\overline{2}}_{p,q}=\Q$ if $n=2,3$.

\begin{thm}\label{maintheo_q}
In the degenerate and non-degenerate geometric algebras $\C$, we have
\begin{eqnarray*}
&\Q^{\overline{1}}=\Gamma^{\overline{1}},\quad\Q^{\overline{2}}=\Gamma^{\overline{2}},\quad \Q^{\overline{3}}=\Gamma^{\overline{3}}, \quad\Q^{\overline{0}}=\Gamma^{\overline{0}},
\end{eqnarray*}
where
\begin{eqnarray}\label{dop_1}
\Gamma^{\overline{1}}\subseteq\Gamma^{\overline{m}}\subseteq\Gamma^{\overline{0}},\qquad m=0,1,2,3.
\end{eqnarray}
\end{thm}
\begin{proof}
For fixed $k=0,1,2,3\mod{4}$ and $m=(k-1)\mod{4}$, $l=(k+1)\mod{4}$, we get
\begin{eqnarray}
    \!\!\!\!\!\!\!\!\!\!\!\!\!\!\!\Gamma^{\overline{k}}&=&\{T\in\C^{\times}:\quad T\C^{\overline{k}}T^{-1}\subseteq(\C^{\overline{km}}\cap\C^{\overline{kl}})\}
    \\
    \!\!\!\!\!\!\!\!\!\!\!\!\!\!\!&=&\{T\in\C^{\times}:\quad T\C^{\overline{k}}T^{-1}\subseteq\C^{\overline{km}},\quad T\C^{\overline{k}}T^{-1}\subseteq\C^{\overline{kl}}\}
    \\
    \!\!\!\!\!\!\!\!\!\!\!\!\!\!\!&=&\{T\in\C^{\times}:\quad\widetilde{T}T\in\Z^{\overline{k}\times},\quad \widehat{\widetilde{T}}T\in\Z^{\overline{k}\times}\}=\Q^{\overline{k}},\label{Q_to_pr}
\end{eqnarray}
where we use Lemma \ref{lemma_for_AB} in the first equality (\ref{Q_to_pr}) and $\Z^{\overline{k}}=\Z^{k}$, $k=1,2,3$, and $\Z^{\overline{0}}=\Z^{4}$ by Lemma \ref{centralizers_qt} in the second equality (\ref{Q_to_pr}). The inclusions in (\ref{dop_1}) follow from the definitions (\ref{def_Q1})--(\ref{def_Q0}) and the statements (\ref{CC3CC4_2}) and (\ref{CC3CC4_3}) of Remark \ref{CC3CC4}.
\end{proof}

\begin{rem}\label{rem_cg}
In the case of the small dimensions $n\leq 4$, the Clifford groups (\ref{def_cg}) coincide with the groups ${\Q}^{\overline{1}}$ because of $\C^{\overline{1}}=\C^{1}$:
\begin{eqnarray}
\Gamma={\Q}^{\overline{1}},\qquad n\leq 4.
\end{eqnarray}
That is why the groups $\Q^{\overline{k}}$ can be considered as generalizations of Clifford groups.
\end{rem}

\begin{rem}\label{rem_Q_AB}
Note that the groups $\Q^{\overline{0}}$ \eqref{def_Q0}, $\Q^{\overline{1}}$ \eqref{def_Q1}, $\Q^{\overline{2}}$ \eqref{def_Q2}, and $\Q^{\overline{3}}$ \eqref{def_Q3} are, by definitions, closely related to the groups $\A^{\overline{01}}$ \eqref{def_A01}, $\B^{\overline{12}}$ \eqref{def_B12}, $\A^{\overline{23}}$ \eqref{def_A23}, and $\B^{\overline{03}}$ \eqref{def_B03} considered in the paper \cite{AB_aaca} (see Section \ref{section_relworks}). Specifically, we have
    \begin{eqnarray}
        &\Q^{\overline{1}} = (\A^{\overline{01}}\cap\B^{\overline{12}}),\qquad (\A^{\overline{23}}\cap\B^{\overline{03}})\subseteq\Q^{\overline{3}},
        \\
        &(\A^{\overline{01}}\cap\B^{\overline{03}})\subseteq\Q^{\overline{0}},\qquad (\A^{\overline{23}}\cap\B^{\overline{12}})\subseteq\Q^{\overline{2}},
    \end{eqnarray}
where for the group $\Q^{\overline{0}}$, we use that $\Z^{1}\subseteq\Z^4$ and $\Z^3\subseteq\Z^4$, and for the group $\Q^{\overline{2}}$, we use that $\Z^1\subseteq\Z^2$ (see Remark \ref{CC3CC4}).
\end{rem}

\begin{lem}\label{lemma_QABP}
    We have the following connections between the groups $\Q^{\overline{1}}$ and $\Q^{\overline{2}}$ and the groups $\P$ \eqref{prev_gr_0} and $\P^{\Lambda}$ \eqref{prev_gr_0_0}:
    \begin{eqnarray}
    &\Q^{\overline{1}}\subseteq\P,\qquad \Q^{\overline{2}}\subseteq\P^{\Lambda}.\label{lemma_QABP_2}
    \end{eqnarray}
\end{lem}
\begin{proof}
 Suppose $T\in\Q^{\overline{1}}$; then $\widetilde{T}T=W\in\Z^{1\times}$ and $\widehat{\widetilde{T}}T=U\in\Z^{1\times}$. Hence, $T=\widetilde{T^{-1}}W$ and $\widehat{T^{-1}}=\widehat{U^{-1}}\widetilde{T}$. We obtain $\widehat{T^{-1}}T=(\widehat{U^{-1}}\widetilde{T})(\widetilde{T^{-1}}W)=\widehat{U^{-1}}W\in\Z^{1\times}=\ker(\ad)$. Therefore, $T\in\P$ by Theorem 4.8 \cite{OnSomeLie}.

   Suppose $T\in\Q^{\overline{2}}$; then $\widetilde{T}T=W\in\Z^{2\times}$ and $\widehat{\widetilde{T}}T=U\in\Z^{2\times}$. Thus, $\widehat{T^{-1}}T=\widehat{U^{-1}}W\in\Z^{2\times}=(\Lambda_r\oplus\C^{n})^{\times}$, so $T\in\P^{\Lambda}$ by Theorem 4.7~\cite{OnSomeLie}.
\end{proof}

\begin{rem}\label{remark_QG}
Using Lemma \ref{lemma_QABP}, Remark \ref{rem_Q_AB}, Theorem \ref{maintheo_q}, Theorem 6.1
\cite{OnSomeLie}, and Theorem 4.1 \cite{AB_aaca}, we obtain
\begin{eqnarray*}
    \!\!\!\!\!\!&\Gamma^{\overline{1}}=(\Gamma^{\overline{01}}\cap\Gamma^{\overline{12}})\subseteq\Gamma^{(1)},
    \qquad \Gamma^{\overline{1}}\subseteq \Gamma^{\overline{2}}\subseteq\Gamma^{(0)}=\tilde{\Gamma}^{(0)},
    \\
    \!\!\!\!\!\! &(\Gamma^{\overline{01}}\cap\Gamma^{\overline{03}}\subseteq\Gamma^{\overline{0}},\qquad(\Gamma^{\overline{23}}\cap\Gamma^{\overline{12}})\subseteq\Gamma^{\overline{2}},\qquad(\Gamma^{\overline{23}}\cap\Gamma^{\overline{03}})\subseteq\Gamma^{\overline{3}},
\end{eqnarray*}
where $\Gamma^{(1)}$ is the group preserving the odd subspace $\C^{(1)}$ under $\ad$ and  $\Gamma^{(0)}=\tilde{\Gamma}^{(0)}$ is the group preserving  $\C^{(0)}$ under $\ad$ and $\tilde{\ad}$ (see (\ref{prev_gr})--(\ref{prev_gr_2}) and the notes above; these groups are considered in \cite{OnSomeLie}); $\Gamma^{\overline{01}}$, $\Gamma^{\overline{12}}$, $\Gamma^{\overline{23}}$, and $\Gamma^{\overline{03}}$ are the groups preserving the subspaces $\C^{\overline{12}}$, $\C^{\overline{23}}$, and $\C^{\overline{03}}$ respectively under $\ad$ (see \eqref{prev_gr_3} and the notes above; these groups are considered in \cite{AB_aaca}). 
\end{rem}

\subsection{The Groups $\check{\Q}^{\overline{0}}$, $\check{\Q}^{\overline{1}}$, $\check{\Q}^{\overline{2}}$, $\check{\Q}^{\overline{3}}$, $\check{\Gamma}^{\overline{0}}$, $\check{\Gamma}^{\overline{1}}$, $\check{\Gamma}^{\overline{2}}$, and $\check{\Gamma}^{\overline{3}}$}\label{section_chQ}

Consider the groups $\check{\Q}^{\overline{1}}$, $\check{\Q}^{\overline{2}}$, $\check{\Q}^{\overline{3}}$, and $\check{\Q}^{\overline{0}}$:
\begin{eqnarray}
    \!\!\!\!\!\!\!\!\!\!\!\!\!\!\!\!\!\!\!\!\!\!\!\!\!\!\!\!\!\!\!\!\!\!\check{\Q}^{\overline{1}}\!\!\!\!\!&:=&\!\!\!\!\!\check{\Q}^{\overline{1}}_{p,q,r}:=\{T\in\C^{\times}:\;\widetilde{T}T\in\check{\Z}^{1\times}=\ker(\tilde{\ad}),\;\widehat{\widetilde{T}}T\in\check{\Z}^{1\times}=\ker(\tilde{\ad})\}\label{def_chQ1}
    \\
    \!\!\!\!\!\!\!\!\!\!\!\!\!&=& \psi^{-1}(\check{\Z}^{1\times})\cap\chi^{-1}(\check{\Z}^{1\times}),
    \\
    \!\!\!\!\!\!\!\!\!\!\!\!\!\!\!\!\!\!\!\!\!\!\!\!\!\!\!\!\!\!\!\!\!\!\check{\Q}^{\overline{2}}\!\!\!\!\!&:=&\!\!\!\!\!\check{\Q}^{\overline{2}}_{p,q,r}:=\{T\in\C^{\times}:\quad\widetilde{T}T\in\check{\Z}^{2\times},\quad \widehat{\widetilde{T}}T\in\check{\Z}^{2\times}\}\label{def_chQ2}
    \\
  \!\!\!\!\!\!\!\!\! \!\!\!\!\!\! \!\!\!\!&=& \psi^{-1}(\check{\Z}^{2\times})\cap\chi^{-1}(\check{\Z}^{2\times}),
    \\
    \!\!\!\!\!\!\!\!\!\!\!\!\!\!\!\!\!\!\!\!\!\!\!\!\!\!\!\!\!\!\!\!\!\!\check{\Q}^{\overline{3}}\!\!\!\!\!&:=&\!\!\!\!\!\check{\Q}^{\overline{3}}_{p,q,r}:=\{T\in\C^{\times}:\quad\widetilde{T}T\in\check{\Z}^{3\times},\quad\widehat{\widetilde{T}}T\in\check{\Z}^{3\times}\}\label{def_chQ3}
    \\
   \!\!\! \!\!\!\!\!\!\!\!\!\!\!\!\!\!\!\!&=& \psi^{-1}(\check{\Z}^{3\times})\cap\chi^{-1}(\check{\Z}^{3\times}),
    \\
    \!\!\!\!\!\!\!\!\!\!\!\!\!\!\!\!\!\!\!\!\!\!\!\!\!\!\!\!\!\!\!\!\!\!\check{\Q}^{\overline{0}}\!\!\!\!\!&:=&\!\!\!\!\!\check{\Q}^{\overline{0}}_{p,q,r}:=\{T\in\C^{\times}:\quad\widetilde{T}T\in\langle\Z^{4}\rangle_{(0)}^{\times},\quad\widehat{\widetilde{T}}T\in\langle\Z^{4}\rangle_{(0)}^{\times}\}\label{def_chQ0}
    \\
   \!\!\! \!\!\!\!\!\!\!\!\!\!\!\!\!\!\!\!&=& \psi^{-1}(\langle\Z^{4}\rangle_{(0)}^{\times})\cap\chi^{-1}(\langle\Z^{4}\rangle_{(0)}^{\times},
\end{eqnarray}
where $\ker(\tilde{\ad})$ (\ref{ker_tilde_ad}) is the kernel of the twisted adjoint representation $\tilde{\ad}$ (\ref{twa22}), the sets $\check{\Z}^2$  and $\check{\Z}^3$ (\ref{def_chZm}) are the twisted centralizers of the subspaces $\C^{2}$ and $\C^{3}$ respectively, and $\Z^4$ is the centralizer of the subspace $\C^{4}$ (see Section \ref{section_centralizers}). By \cite{centr}, we have
\begin{eqnarray*}
    \!\!\!\!\!\!\!\!\!\!\!\!\!\!\!\!\!\!&&\check{\Z}^1=\Lambda,\label{ch_cc_1}
    \\
    \!\!\!\!\!\!\!\!\!\!\!\!\!\!\!\!\!\!&&\check{\Z}^2=
    \left\lbrace
    \begin{array}{lll}\label{ch_cc_2}
    \Lambda^{(0)}\oplus\Lambda^{n}\oplus \{\C^{1}_{p,q}\Lambda^{n-1}\},&&\mbox{$n$ is odd},
    \\
    \Lambda^{(0)}\oplus\Lambda^{n-1}\oplus \{\C^{1}_{p,q}\Lambda^{n-2}\}\oplus\C^{n},&&\mbox{$n$ is even},\;\; r\neq n,
    \\
    \Lambda^{(0)}\oplus\Lambda^{n-1},&&\mbox{$n$ is even},\;\; r=n,
    \end{array}
    \right.
    \\
    \!\!\!\!\!\!\!\!\!\!\!\!\!\!\!\!\!\!&&\check{\Z}^3=\Lambda\oplus \{\C^{1}_{p,q}\Lambda^{\geq n-2}\}\oplus \{\C^{2}_{p,q}\Lambda^{\geq n-3}\},\label{ch_cc_3}
 \end{eqnarray*}
 and
 \begin{eqnarray*}
 \langle\Z^{4}\rangle_{(0)}\!=
\!    \left\lbrace
    \begin{array}{lll}
\!\!\!\Lambda^{(0)}\oplus \{\C^{1}_{p,q}\Lambda^{n-2}\}\oplus \{\C^{2}_{p,q}\Lambda^{n-3}\},\!\!\!\!\!\!\!\!\!&&\mbox{$n$ is odd},
\\
\!\!\!\Lambda^{(0)}\oplus \{\C^{1}_{p,q}\Lambda^{n-3}\}\oplus \{\C^{2}_{p,q}\Lambda^{n-4}\} \oplus \C^{n}, \!\!\!\!\!\!\!\!\!&&\mbox{$n$ is even},\;r\neq n,
\\
\!\!\!\Lambda^{(0)}, \!\!\!\!\!\!\!\!\!&&\mbox{$n$ is even},\; r=n.
    \end{array}
    \right.
\end{eqnarray*}
We use $\check{\;}$ in the notation of the groups (\ref{def_chQ1})--(\ref{def_chQ0}) due to Theorem \ref{maintheo_checkq} below.

In the special case of the non-degenerate geometric algebra $\C_{p,q}$, we obtain the groups $\Q^{\pm}$ and $\Q'$ considered in \cite{GenSpin}:
\begin{eqnarray*}
    \check{\Q}^{\overline{1}}_{p,q}=\check{\Q}^{\overline{3}}_{p,q}=\Q^{\pm},\quad
     \check{\Q}^{\overline{0}}_{p,q}=\check{\Q}^{\overline{2}}_{p,q}=
     \left\lbrace
    \begin{array}{lll}
    \Q^{\pm},&&\!\!\!\!\!\!\!\!\!\!n=1,2,3\mod{4},
    \\
    \Q',&&\!\!\!\!\!\!\!\!\!\!n=0\mod{4},
    \end{array}
    \right.
    \; n\neq1,2,
\end{eqnarray*}
and $\check{\Q}^{\overline{1}}_{p,q}=\check{\Q}^{\overline{0}}_{p,q}=\Q^{\pm}$ if $n=1,2$.

\begin{thm}\label{maintheo_checkq}
In degenerate and non-degenerate geometric algebras $\C$, we have
\begin{eqnarray*}
&\check{\Q}^{\overline{1}}=\check{\Gamma}^{\overline{1}}\subseteq\check{\Q}^{\overline{3}}=\check{\Gamma}^{\overline{3}},\qquad \check{\Q}^{\overline{2}}=\check{\Gamma}^{\overline{2}},\qquad \check{\Q}^{\overline{0}}=\check{\Gamma}^{\overline{0}}.
\end{eqnarray*}
\end{thm}
\begin{proof}
We get $\check{\Q}^{\overline{1}}\subseteq\check{\Q}^{\overline{3}}$, using $\check{\Z}^1\subseteq\check{\Z}^3$ by Remark \ref{CC3CC4}.
    For fixed $k=0,1,2,3\mod{4}$ and $m=(k-1)\mod{4}$, $l=(k+1)\mod{4}$, we obtain
\begin{eqnarray}
    \!\!\!\!\!\!\!\!\!\!\!\!\!\!\!\check{\Gamma}^{\overline{k}}&=&\{T\in\C^{\times}:\quad \widehat{T}\C^{\overline{k}}T^{-1}\subseteq(\C^{\overline{km}}\cap\C^{\overline{kl}})\}
    \\
    \!\!\!\!\!\!\!\!\!\!\!\!\!\!\!&=&\{T\in\C^{\times}:\quad \widehat{T}\C^{\overline{k}}T^{-1}\subseteq\C^{\overline{km}},\quad \widehat{T}\C^{\overline{k}}T^{-1}\subseteq\C^{\overline{kl}}\}
    \\
    \!\!\!\!\!\!\!\!\!\!\!\!\!\!\!&=&\{T\in\C^{\times}:\quad\widetilde{T}T\in\check{\Z}^{\overline{k}\times},\quad \widehat{\widetilde{T}}T\in\check{\Z}^{\overline{k}\times}\}=\check{\Q}^{\overline{k}},\label{Q_to_pr_2}
\end{eqnarray}
 where we use Lemma \ref{lemma_for_AB} in the first equality (\ref{Q_to_pr_2}) and $\check{\Z}^{\overline{k}}=\check{\Z}^{k}$, $k=1,2,3$, and $\check{\Z}^{\overline{0}}=\check{\Z}^{4}\cap\C^{(0)}={\Z}^{4}\cap\C^{(0)}$ by Lemma \ref{centralizers_qt} and (\ref{same_even}) in the second equality (\ref{Q_to_pr_2}).
\end{proof}

\begin{rem}\label{rem_lg}
In the case of the small dimensions $n\leq 4$, the Lipschitz groups (\ref{def_lg}) coincide with the groups $\check{\Q}^{\overline{1}}$ due to $\C^{\overline{1}}=\C^{1}$:
\begin{eqnarray}
\Gamma^{\pm}=\check{\Q}^{\overline{1}},\qquad n\leq 4.
\end{eqnarray}
That is why the groups $\check{\Q}^{\overline{k}}$ can be considered as generalizations of the Lipschitz groups.
\end{rem}

\begin{rem}
    The groups $\check{\Q}^{\overline{1}}$ \eqref{def_chQ1}, $\check{\Q}^{\overline{2}}$  \eqref{def_chQ2}, $\check{\Q}^{\overline{3}}$ \eqref{def_chQ3}, and $\check{\Q}^{\overline{0}}$ \eqref{def_chQ0} are closely related to the groups $\check{\A}^{\overline{12}}$ \eqref{def_chA12}, $\check{\A}^{\overline{03}}$ \eqref{def_chA03}, $\check{\B}^{\overline{01}}$ \eqref{def_chB01}, and $\check{\B}^{\overline{23}}$ \eqref{def_chB23} considered in detail in \cite{AB_aaca} (see Section \ref{section_relworks}):
    \begin{eqnarray}
        &(\check{\A}^{\overline{12}}\cap\check{\B}^{\overline{01}})\subseteq \check{\Q}^{\overline{1}},\qquad (\check{\A}^{\overline{12}}\cap\check{\B}^{\overline{23}})\subseteq \check{\Q}^{\overline{2}},
        \\
        & (\check{\A}^{\overline{03}}\cap\check{\B}^{\overline{23}})\subseteq \check{\Q}^{\overline{3}},\qquad (\check{\A}^{\overline{03}}\cap\check{\B}^{\overline{01}})\subseteq \check{\Q}^{\overline{0}},
    \end{eqnarray}
    where for $\check{\Q}^{\overline{3}}$, we use that $\Z^3\cap\C^{(0)} = \langle \Z^3\rangle_{(0)} = \langle \check{\Z}^3\rangle_{(0)}\subseteq \check{\Z}^3$, and for  $\check{\Q}^{\overline{0}}$, we use that $\langle \Z^3\rangle_{(0)}\subseteq \langle \Z^4\rangle_{(0)}$ and $\langle \Z^1\rangle_{(0)}\subseteq \langle \Z^4\rangle_{(0)}$ (see Remark \ref{CC3CC4} and  formula \eqref{same_even}).
\end{rem}

\begin{lem}\label{lemma_chQ_eq}
We have the following inclusions:
\begin{eqnarray}
\!\!\!\!\!\!\!\!\!\!\!\!&\check{\Q}^{\overline{1}}\subseteq\P^{\pm\Lambda},\quad \check{\Q}^{\overline{m}}\subseteq\Q^{\overline{0}},\quad m=0,1,2,3,\quad \check{\Q}^{\overline{1}}\subseteq\Q^{\overline{2}},\label{sub_ch_Q}
\end{eqnarray}
where $\P^{\pm\Lambda}$ is the group defined in formula \eqref{prev_gr_0_0} and considered in \cite{OnSomeLie}.
\end{lem}
\begin{proof}
    Let us prove $\check{\Q}^{\overline{1}}\subseteq\P^{\pm\Lambda}$. Suppose $T\in\check{\Q}^{\overline{1}}$; then  $\widetilde{T}T=W\in\ker(\tilde{\ad})$ and $\widehat{\widetilde{T}}T=U\in\ker(\tilde{\ad})$. Therefore, $T=\widetilde{T^{-1}}W$ and $\widehat{T^{-1}}=\widehat{U^{-1}}\widetilde{T}$. Thus, $\widehat{T^{-1}}T=(\widehat{U^{-1}}\widetilde{T})(\widetilde{T^{-1}}W)=\widehat{U^{-1}}W\in\ker(\tilde{\ad})$ and  $T\in\P^{\pm\Lambda}$ by Theorem 4.7 \cite{OnSomeLie}. 
    
    Other inclusions in (\ref{sub_ch_Q}) follow from the definitions of the groups $\check{\Q}^{\overline{m}}$ (\ref{def_chQ1})--(\ref{def_chQ0}), $m=0,1,2,3$, $\Q^{\overline{0}}$ (\ref{def_Q0}), $\Q^{\overline{2}}$ (\ref{def_Q2}), $\check{\Z}^1\subseteq\Z^2$, and Remark~\ref{CC3CC4}. 
\end{proof}

\begin{rem}
    Using Lemma \ref{lemma_chQ_eq}, Theorem \ref{maintheo_checkq}, Remark \ref{remark_QG}, Theorem 6.1 \cite{OnSomeLie}, and Theorem 4.2 \cite{AB_aaca}, we obtain
    \begin{eqnarray*}
        \!\!\!\!&(\check{\Gamma}^{\overline{12}}\cap\check{\Gamma}^{\overline{01}})\subseteq\check{\Gamma}^{\overline{1}}\subseteq\check{\Gamma}^{(1)}=\tilde{\Gamma}^{(1)},\qquad \check{\Gamma}^{\overline{1}}\subseteq\Gamma^{\overline{2}}\subseteq\Gamma^{(0)}=\tilde{\Gamma}^{(0)},
        \\
        \!\!\!\!&
        \check{\Gamma}^{\overline{m}}\subseteq\Gamma^{\overline{0}},\qquad m=0,1,2,3;
        \qquad(\check{\Gamma}^{\overline{12}}\cap\check{\Gamma}^{\overline{23}})\subseteq\check{\Gamma}^{\overline{2}},
    \end{eqnarray*}
    where $\check{\Gamma}^{(1)}=\tilde{\Gamma}^{(1)}$ is the group preserving $\C^{(1)}$ under $\check{\ad}$ and $\tilde{\ad}$ and $\Gamma^{(0)}=\tilde{\Gamma}^{(0)}$ is the group preserving $\C^{(0)}$ under $\ad$ and $\tilde{\ad}$ (see (\ref{prev_gr})--(\ref{prev_gr_2}) and the notes above; these groups are considered in \cite{OnSomeLie}); $\check{\Gamma}^{\overline{12}}$, $\check{\Gamma}^{\overline{01}}$, and $\check{\Gamma}^{\overline{23}}$ are the groups preserving the subspaces $\C^{\overline{12}}$, $\C^{\overline{01}}$, and $\C^{\overline{23}}$ respectively under $\check{\ad}$ (see \eqref{prev_gr_4} and the notes above; these groups are considered in \cite{AB_aaca}). 
\end{rem}

\section{Examples on the Considered Groups}\label{section_examples}

\begin{rem}
    In the case of the Grassmann algebra $\C_{0,0,n}$, we have
    \begin{eqnarray*}
\Lambda^{\times}_{n}=\Q^{\overline{2}}_{0,0,n}=\Q^{\overline{0}}_{0,0,n}=\check{\Q}^{\overline{1}}_{0,0,n}=\check{\Q}^{\overline{3}}_{0,0,n}
    \end{eqnarray*}
and
    \begin{eqnarray}
        \!\!\!\!\!\!\!\!\!\!\Q^{\overline{1}}_{0,0,n}\!\!\!\!&=&\!\!\!\!\{T\in\Lambda^{\times}_n:\;\; \widetilde{T}T\in\ker(\ad),\;\; \widehat{\widetilde{T}}T\in\ker(\ad)\},
        \\
        \!\!\!\!\!\!\!\!\!\!\check{\Q}^{\overline{0}}_{0,0,n}\!\!\!\!&=&
        \!\!\!\!\{T\in\Lambda^{\times}_n:\;\;\widetilde{T}T\in\ker(\check{\ad}),\;\;\widehat{\widetilde{T}}T\in\ker(\check{\ad})\}.
        \end{eqnarray}
        If $n$ is odd, then
        \begin{eqnarray}
        \!\!\!\!\!\!\!\!\!\!\Q^{\overline{3}}_{0,0,n}\!\!\!\!&=&
        \!\!\!\!\{T\in\Lambda^{\times}_n:\;\; \widetilde{T}T\in(\Lambda^{(0)n}_n\oplus\Lambda^{n-2}_n)^{\times},\;\; \widehat{\widetilde{T}}T\in(\Lambda^{(0)n}_n\oplus\Lambda^{n-2}_n)^{\times}\},
        \\
        \!\!\!\!\!\!\!\!\!\!\check{\Q}^{\overline{2}}_{0,0,n}\!\!\!\!&=&\!\!\!\!\{T\in\Lambda^{\times}_n:\;\; \widetilde{T}T\in\Lambda^{(0)n\times}_n,\;\;\widehat{\widetilde{T}}T\in\Lambda^{(0)n\times}_n\}.
    \end{eqnarray}
           If $n$ is even, then
        \begin{eqnarray*}
        \Q^{\overline{3}}_{0,0,n}=\check{\Q}^{\overline{2}}_{0,0,n}=
        \{T\in\Lambda^{\times}_n:\; \widetilde{T}T\in(\Lambda^{(0)}_n\oplus\Lambda^{ n-1}_n)^{\times},\; \widehat{\widetilde{T}}T\in(\Lambda^{(0)}_n\oplus\Lambda^{n-1}_n)^{\times}\},
    \end{eqnarray*}
and
    \begin{eqnarray}
        & \Q^{\overline{1}}_{0,0,n}=\check{\Q}^{\overline{0}}_{0,0,n}.
    \end{eqnarray}
\end{rem}

\begin{rem}\label{small_dim}
Below, we write out the groups $\Q^{\overline{k}}$ and $\check{\Q}^{\overline{k}}$, $k=0,1,2,3$, considered in this paper as well as the groups $\P^{\pm}$, $\P$, $\P^{\pm\Lambda}$, $\P^{\Lambda}$ (\ref{prev_gr})--(\ref{prev_gr_2}) \cite{OnSomeLie} and 
$\A^{\overline{01}}$, $\A^{\overline{23}}$, $\check{\A}^{\overline{12}}$, $\check{\A}^{\overline{03}}$, $\B^{\overline{12}}$, $\B^{\overline{03}}$, $\check{\B}^{\overline{01}}$, $\check{\B}^{\overline{23}}$ \eqref{prev_gr_3}--\eqref{prev_gr_4} \cite{AB_aaca} in the case of low dimensions. These groups  are considered in \cite{OnInner,GenSpin} in the case of the low-dimensional ($n\leq6$) non-degenerate algebra $\C_{p,q}$.

    If $n=1$, then we have two different groups in the case of the non-degenerate algebra $\C_{p,q}$:
    \begin{eqnarray*}
        \C^{\times}_{p,q}&=&\P_{p,q}=\P^{\Lambda}_{p,q}=\Q^{\overline{0}}_{p,q}=\Q^{\overline{1}}_{p,q}=\Q^{\overline{2}}_{p,q}=\Q^{\overline{3}}_{p,q}
        \\
        &=&\check{\Q}^{\overline{2}}_{p,q}=\check{\Q}^{\overline{3}}_{p,q} = \A^{\overline{01}}_{p,q} = \A^{\overline{23}}_{p,q} = \B^{\overline{12}}_{p,q} = \B^{\overline{03}}_{p,q}
        \\
        &=&\check{\B}^{\overline{01}}_{p,q} = \check{\B}^{\overline{23}}_{p,q};
        \\
        \C^{(0)\times}_{p,q}\cup\C^{(1)\times}_{p,q}&=&\P^{\pm}_{p,q}=\P^{\pm\Lambda}_{p,q}=\check{\Q}^{\overline{0}}_{p,q}=\check{\Q}^{\overline{1}}_{p,q} = \check{\A}^{\overline{12}}_{p,q} = \check{\A}^{\overline{03}}_{p,q};
    \end{eqnarray*}
    as well as in the case of the Grassmann algebra $\C_{0,0,1}$:
        \begin{eqnarray*}
        \C^{\times}_{0,0,1}&=&\P_{0,0,1}=\P^{\pm\Lambda}_{0,0,1}=\P^{\Lambda}_{0,0,1}=\Q^{\overline{0}}_{0,0,1}=\Q^{\overline{1}}_{0,0,1}=\Q^{\overline{2}}_{0,0,1}
        \\
        &=&\Q^{\overline{3}}_{0,0,1}=\check{\Q}^{\overline{2}}_{0,0,1}=\check{\Q}^{\overline{3}}_{0,0,1}=\check{\Q}^{\overline{1}}_{0,0,1}=\B^{\overline{12}}_{0,0,1} =  \B^{\overline{03}}_{0,0,1}
        \\
        &=&  \check{\B}^{\overline{01}}_{0,0,1}=\check{\B}^{\overline{23}}_{0,0,1} =\A^{\overline{01}}_{0,0,1}=\A^{\overline{23}}_{0,0,1}=\check{\A}^{\overline{12}}_{0,0,1};
        \\
        \C^{0\times}_{0,0,1}&=&\P^{\pm}_{0,0,1}=\check{\Q}^{\overline{0}}_{0,0,1}=\check{\A}^{\overline{03}}_{0,0,1}.
    \end{eqnarray*}
    If $n=2$, then we have two different groups in the case of the non-degenerate algebra $\C_{p,q}$:
    \begin{eqnarray*}
    \C^{\times}_{p,q}&=&\Q^{\overline{3}}_{p,q}=\Q^{\overline{0}}_{p,q}=\check{\Q}^{\overline{2}}_{p,q}=\check{\Q}^{\overline{3}}_{p,q} 
    \\
    &=& \B^{\overline{12}}_{p,q} = \B^{\overline{03}}_{p,q}=\check{\B}^{\overline{01}}_{p,q}=\check{\B}^{\overline{23}}_{p,q};
    \\
    \C^{(0)\times}_{p,q}\cup\C^{(1)\times}_{p,q}&=&\P^{\pm}_{p,q}=\P_{p,q}=\P^{\pm\Lambda}_{p,q}=\P^{\Lambda}_{p,q}=\Q^{\overline{1}}_{p,q}=\check{\Q}^{\overline{0}}_{p,q}
    \\
    &=&\Q^{\overline{2}}_{p,q}=\check{\Q}^{\overline{1}}_{p,q} = \A^{\overline{01}}_{p,q} = \A^{\overline{23}}_{p,q}= \check{\A}^{\overline{12}}_{p,q} = \check{\A}^{\overline{03}}_{p,q};
    \end{eqnarray*}
    and the Grassmann algebra $\C_{0,0,2}$:
    \begin{eqnarray*}
    \C^{\times}_{0,0,2}&=&\P^{\pm\Lambda}_{0,0,2}=\P^{\Lambda}_{0,0,2}=\Q^{\overline{0}}_{0,0,2}=\Q^{\overline{2}}_{0,0,2}=\Q^{\overline{3}}_{0,0,2}=\check{\Q}^{\overline{1}}_{0,0,2}
    \\
    &=&\check{\Q}^{\overline{2}}_{0,0,2}=\check{\Q}^{\overline{3}}_{0,0,2}=\B^{\overline{12}}_{0,0,2} =  \B^{\overline{03}}_{0,0,2} =\check{\B}^{\overline{01}}_{0,0,2}=\check{\B}^{\overline{23}}_{0,0,2}
    \\
    &=&  \check{\A}^{\overline{12}}_{0,0,2}=\A^{\overline{23}}_{0,0,2};
    \\
    \C^{(0)\times}_{0,0,2}&=&\P^{\pm}_{0,0,2}=\P_{0,0,2}=\Q^{\overline{1}}_{0,0,2}=\check{\Q}^{\overline{0}}_{0,0,2}=\A^{\overline{01}}_{0,0,2}=\check{\A}^{\overline{03}}_{0,0,2};
    \end{eqnarray*}
    and three different groups in the case of $\C_{p,q,1}$ (with $(p,q)=(1,0)$ or $(0,1)$): 
    \begin{eqnarray*}
\C^{\times}_{p,q,1}&=&\Q^{\overline{3}}_{p,q,1}=\Q^{\overline{0}}_{p,q,1}=\check{\Q}^{\overline{2}}_{p,q,1}=\check{\Q}^{\overline{3}}_{p,q,1}
\\
&=&\B^{\overline{12}}_{p,q,1}=\B^{\overline{03}}_{p,q,1}=\check{\B}^{\overline{01}}_{p,q,1}=\check{\B}^{\overline{23}}_{p,q,1};
        \\
(\C^{(0)\times}_{p,q,1}\cup\C^{(1)\times}_{p,q,1})\Lambda^{\times}_1&=&\P^{\pm\Lambda}_{p,q,1}=\P^{\Lambda}_{p,q,1}=\Q^{\overline{2}}_{p,q,1}=\check{\Q}^{\overline{1}}_{p,q,1} = \A^{\overline{23}}_{p,q,1}=\check{\A}^{\overline{12}}_{p,q,1};
        \\
\C^{(0)\times}_{p,q,1}\cup\C^{(1)\times}_{p,q,1}&=&\P^{\pm}_{p,q,1}=\P_{p,q,1}=\Q^{\overline{1}}_{p,q,1}=\check{\Q}^{\overline{0}}_{p,q,1}=\A^{\overline{01}}_{p,q,1}=\check{\A}^{\overline{03}}_{p,q,1}.
    \end{eqnarray*}
\end{rem}

Let us provide some examples on the connections between the groups considered in Section \ref{section_qt} and the well-known matrix groups. We use that any degenerate geometric algebra $\C_{p,q,r}$, $r\neq0$, can be embedded into the non-degenerate algebra of larger dimension (see Clifford--Jordan--Wigner representation in, for example, \cite{CJW2}). Any non-degenerate geometric algebra is isomorphic to some matrix algebra by the Cartan's periodicity theorem (see, for example, \cite{lounesto,p}).
Let us consider the groups of upper triangular matrices $\UT(2,\F)$ and $\UT(4,\F)$ \cite{baker}:
\begin{eqnarray}
\UT(2,\F)&:=&\left\{\begin{bmatrix}
x_{11} & x_{12} \\
0 & x_{21} \\
\end{bmatrix}\in\GL(2,\F)\right\},\label{ut_2}
\\
\UT(4,\F)&:=&\left\{\begin{bmatrix}
x_{11} & x_{12} & x_{13} & x_{14}\\
0 & x_{22} & x_{23} & x_{24} \\
0 & 0 & x_{33} & x_{34} \\
0 & 0 & 0 & x_{44}\\
\end{bmatrix}\in\GL(4,\F)\right\},
\label{ut_4}
\end{eqnarray}
and a unipotent subgroup $\SUT(2,\F)$ \cite{baker} of $\UT(2,\F)$:
\begin{eqnarray}
\SUT(2,\F):=\left\{\begin{bmatrix}
1 & x_{12} \\
0 & 1 \\
\end{bmatrix},\; x_{12}\in\F\right\}.\label{sut}
\end{eqnarray}
Also consider the higher-dimensional Heisenberg group $\Heis_4$ (see, for example, \cite{baker}):
\begin{eqnarray}\label{heis}
\Heis_4=
\left\{\begin{bmatrix}
1 & x_{12} & x_{13} & x_{14}\\
0 & 1 & 0 & x_{24} \\
0 & 0 & 1 & x_{34} \\
0 & 0 & 0 & 1\\
\end{bmatrix}\in\GL(4,\F)\right\}.
\end{eqnarray}

\begin{ex}\label{ex_mat_1}
Consider the Grassmann algebra $\Lambda_{1}=\C_{0,0,1}$, which can be embedded into the non-degenerate algebra $\C_{1,1,0}\cong\Mat(2,\F)$. For $e$ and $e_1$ from $\Lambda_{1}$, we have:
\begin{eqnarray}
e\mapsto \begin{bmatrix} 
1 & 0 \\ 
0 & 1 \\
\end{bmatrix}, \quad 
e_1\mapsto \begin{bmatrix} 
0 & 1 \\ 
0 & 0 \\
\end{bmatrix}.
\end{eqnarray}
Also consider the Grassmann algebra $\Lambda_{2}=\C_{0,0,2}$, which can be embedded into $\C_{2,2,0}\cong\Mat(4,\F)$. For the basis elements $e_1$, $e_2$, and $e_{12}$ of $\Lambda_{2}$, we have:
\begin{eqnarray*}
e_1\mapsto \begin{bmatrix} 
0 & 1 & 0 & 0 \\ 
0 & 0 & 0 & 0 \\
0 & 0 & 0 & 1 \\
0 & 0 & 0 & 0 \\
\end{bmatrix},\quad 
e_2\mapsto \begin{bmatrix} 
0 & 0 & 1 & 0 \\ 
0 & 0 & 0 & -1 \\
0 & 0 & 0 & 0 \\
0 & 0 & 0 & 0 \\
\end{bmatrix},\quad
e_{12}\mapsto \begin{bmatrix} 
0 & 0 & 0 & -1 \\ 
0 & 0 & 0 & 0 \\
0 & 0 & 0 & 0 \\
0 & 0 & 0 & 0 \\
\end{bmatrix}.
\end{eqnarray*}
Note that
\begin{eqnarray*}
    \Lambda^{(0)}_2\cong 
\left\{\begin{bmatrix} 
x_0 & 0 & 0 & x_1 \\ 
0 & x_0 & 0 & 0 \\
0 & 0 & x_0 & 0 \\
0 & 0 & 0 & x_0 \\
\end{bmatrix},\;\; x_0,x_1\in\F\right\}\cong
\left\{\begin{bmatrix} 
x_0 & x_1 \\ 
0 & x_0 \\
\end{bmatrix},\;\; x_0,x_1\in\F\right\}.
\end{eqnarray*}
Using Remark \ref{small_dim}, we obtain $\Lambda^{(0)\times}_1=\P^{\pm}_{0,0,1}=\check{\Q}^{\overline{0}}_{0,0,1}\cong\F^{\times}$,
    \begin{eqnarray}\label{l1l02}
        \!\!\!\!\!\!\!\!\!\!\!\!\!\!\!\!\!\!&&\Lambda^{\times}_1=\P_{0,0,1}=\Q^{\overline{0}}_{0,0,1}=\Q^{\overline{1}}_{0,0,1}=\check{\Q}^{\overline{1}}_{0,0,1}\cong\Lambda^{(0)\times}_2=\P^{\pm}_{0,0,2}=\P_{0,0,2}\label{gr_1}
        \\
        \!\!\!\!\!\!\!\!\!\!\!\!\!\!\!\!\!\!&&=\Q^{\overline{1}}_{0,0,2}=\check{\Q}^{\overline{0}}_{0,0,2}\cong 
        \left\{
        \begin{bmatrix} x_0 & x_1 \\ 
        0 & x_0 \\
\end{bmatrix},\; x_0,x_1\in\F,\; x_0\neq0
        \right\},\label{gr_2}
    \end{eqnarray}
and
\begin{eqnarray}
&&\Lambda^{\times}_2=\P^{\pm\Lambda}_{0,0,2}=\Q^{\overline{0}}_{0,0,2}=\Q^{\overline{2}}_{0,0,2}=\check{\Q}^{\overline{1}}_{0,0,2}=\check{\Q}^{\overline{2}}_{0,0,2}
\\
&&
    \cong\left\{
\begin{bmatrix} 
x_0 & x_1 & x_2 & x_3 \\
0 & x_0 & 0 & - x_2 \\
0 & 0 & x_0 & x_1 \\
0 & 0& 0 & x_0 \\
\end{bmatrix}:\;\; x_0,x_1,x_2,x_3\in\F,\;\; x_0\neq0
\right\}.\label{gr_3}
\end{eqnarray}
Note that in the cases $\Lambda_1$ and $\Lambda_2$, all the groups considered in this work  can be realized as subgroups of the upper triangular matrix groups $\UT(2,\F)$ (\ref{ut_2}) and $\UT(4,\F)$ (\ref{ut_4}).
The unitriangular group $\SUT(2,\F)$ (\ref{sut}) can be realized as a subgroup of the groups (\ref{gr_1})--(\ref{gr_2}).
The unipotent subgroup of (\ref{gr_3}), i.e. the subgroup of (\ref{gr_3}) with $x_0=1$, is a subgroup of the Heisenberg group $\Heis_4$ (\ref{heis}).

\end{ex}

\begin{ex}
    Consider the real degenerate algebra $\C_{0,1,1}(\BR)$, which can be embedded into the real non-degenerate algebra $\C_{1,2,0}(\BR)\cong\Mat(2,\BC)$. For the basis elements $\C_{0,1,1}(\BR)$, we have the following complex matrices:
    \begin{eqnarray}
e\mapsto \begin{bmatrix} 
1 & 0 \\ 
0 & 1 \\
\end{bmatrix}, \quad 
e_1\mapsto \begin{bmatrix} 
i & 0 \\ 
0 & -i \\
\end{bmatrix}, \quad 
e_2\mapsto \begin{bmatrix} 
0 & 1 \\ 
0 & 0 \\
\end{bmatrix}, \quad 
e_{12}\mapsto \begin{bmatrix} 
0 & i \\ 
0 & 0 \\
\end{bmatrix}.
    \end{eqnarray}
Note that (see example above)
\begin{eqnarray}
\C^{(0)\times}_{0,1,1}(\BR)\cong\left \{
    \begin{bmatrix} 
x_0 & x_{1} \\ 
0 & x_0 \\
\end{bmatrix},\; x_0,x_1\in\BR,\; x_0\neq0\right\}\cong\Lambda^{\times}_1(\BR)\cong\Lambda^{(0)\times}_2(\BR).
\end{eqnarray}
We obtain the following isomorphisms:
\begin{eqnarray}
\!\!\!\!\!\!\!\!\!\!\!\!\!\!&&\C^{\times}_{0,1,1}(\BR)=\Q^{\overline{0}}_{0,1,1}=\Q^{\overline{3}}_{0,1,1}=\check{\Q}^{\overline{2}}_{0,1,1}=\check{\Q}^{\overline{3}}_{0,1,1}\label{gr_4}
\\
\!\!\!\!\!\!\!\!\!\!\!\!\!\!&&\cong
\left\{
    \begin{bmatrix} 
x_0 + ix_1 & x_2 + ix_3 \\
0 & x_0-i x_1 \\
\end{bmatrix},\; x_0,x_1,x_2,x_3\in\BR,\;(x_0)^2+(x_1)^2\neq0\right\},\label{gr_4_2}
\\
\!\!\!\!\!\!\!\!\!\!\!\!\!\!&&\C^{(0)\times}_{0,1,1}(\BR)\cup\C^{(1)\times}_{0,1,1}(\BR)=\Q^{\overline{1}}_{0,1,1}=\check{\Q}^{\overline{0}}_{0,1,1}\label{gr_5}
    \\
\!\!\!\!\!\!\!\!\!\!\!\!\!\!&&\cong
    \left\{
    \begin{bmatrix} 
x_0 & i x_{1} \\ 
0 & x_0 \\
\end{bmatrix},
    \begin{bmatrix} 
i x_0 & x_1 \\ 
0 & -ix_0 \\
\end{bmatrix},\; x_0,x_1\in\BR,\; x_0\neq0\right\},
\\
\!\!\!\!\!\!\!\!\!\!\!\!\!\!&&\C^{12\times}_{0,1,1}\cup(\C^{(0)}_{0,1,1}\oplus\Lambda^1_1)^{\times}=\Q^{\overline{2}}_{0,1,1}=\check{\Q}^{\overline{1}}_{0,1,1}\label{gr_6}
\\
\!\!\!\!\!\!\!\!\!\!\!\!\!\!&&\cong
\left\{
    \begin{bmatrix} 
i x_1 & x_2 + ix_3 \\
0 & -ix_1 \\
\end{bmatrix},
    \begin{bmatrix} 
x_1 & x_2 + ix_3 \\
0 & x_1 \\
\end{bmatrix},\; x_1,x_2,x_3\in\BR,\; x_1\neq0\right\}.\label{gr_6_2}
\end{eqnarray}
Note that all the considered groups (\ref{gr_4}), (\ref{gr_5}), and (\ref{gr_6}) can be realized as subgroups of the upper triangular matrix group $\UT(2,\BC)$ (\ref{ut_2}). The unitriangular group $\SUT(2,\BC)$ is a subgroup of the matrix groups (\ref{gr_4_2}) and (\ref{gr_6_2}).

\end{ex}

\section{The Lie Algebras of the Considered Lie Groups}\label{section_liealg}

\begin{thm}\label{thm_alg}
    The Lie algebras of the Lie groups $\Q^{\overline{k}}$ and $\check{\Q}^{\overline{k}}$, $k=0,1,2,3$, are presented in Table \ref{table_lie_alg}.
\end{thm}
\begin{proof}
    To prove the statements, we use the well-known facts on the relation between an arbitrary Lie group and the corresponding Lie algebra. 
    For example, when $n$ is even, differentiating the defining equations of $\Q^{\overline{1}}$ (\ref{def_Q1}) gives $X+\widetilde{X}\in\Z^1=\Lambda^{(0)}$ and $X+\widehat{\widetilde{X}}\in\Z^1=\Lambda^{(0)}$ for $X\in\mathfrak{q}^{\overline{1}}$, and by projecting $X$ onto each $\C^{\overline{m}}$, we find that these equations are satisfied precisely when $X\in\C^{\overline{2}}\oplus\Lambda^{\overline{0}}$.
\end{proof}

\afterpage{
\begin{table}[h]
\caption{The Lie groups and the corresponding Lie algebras\protect\footnotemark}\label{table_lie_alg}
\!\!\begin{tabular}{p{2.15cm}p{1.25cm}p{0.6cm}p{6.4cm}}  \hline
Lie group & $n\mod{4}$ & $r$ & Lie algebra \\ \hline
$\Q^{\overline{1}}$ & $0,2$ & & $\C^{\overline{2}}\oplus\Lambda^{\overline{0}}$ \\ 
$\check{\Q}^{\overline{0}}$ & $3$ & & \\
$\check{\Q}^{\overline{0}}$ & $0$ & $n$ & \\\hline
$\Q^{\overline{1}}$ & $1,3$ & & $\C^{\overline{2}n}\oplus\Lambda^{\overline{0}}$ \\
$\check{\Q}^{\overline{0}}$ & $0$ & $\neq n$ & \\ \hline
$\Q^{\overline{2}}\!=\!\check{\Q}^{\overline{1}}$ & $2$ & & $\C^{\overline{2}}\oplus\Lambda^{\overline{0}(1)}$ \\ 
$\check{\Q}^{\overline{1}}$ & $0,1,3$ & & \\
$\Q^{\overline{2}},\!\Q^{\overline{0}}$ & $0,1,3$ & $n$ &  \\ \hline
$\Q^{\overline{2}}$ & $0,1,3$ & $\neq n$ & $\C^{\overline{2}n}\oplus\Lambda^{\overline{0}(1)}$ \\  \hline
$\Q^{\overline{3}},\!\check{\Q}^{\overline{2}}$ & $2$ & & $\C^{\overline{2}}\oplus\Lambda^{\overline{0}}\oplus\Lambda^{n-1}\oplus\{\C^{1}_{p,q}\Lambda^{n-2}\}$\\ 
$\check{\Q}^{\overline{2}}$ & $0$ & $n$ & \\ \hline
$\Q^{\overline{3}}$ & $3$ & & $\C^{\overline{2}n}\oplus\Lambda^{\overline{0}}\oplus\Lambda^{n-2}\oplus\{\C^{1}_{p,q}\Lambda^{n-3}\}$ \\ \hline
$\check{\Q}^{\overline{2}}$ & $1,3$ & & $\C^{\overline{2}}\oplus\Lambda^{\overline{0}}\oplus\Lambda^{n}\oplus\{\C^{1}_{p,q}\Lambda^{n-1}\}$ \\ \hline
$\check{\Q}^{\overline{2}}$ & $0$ & $\neq n$ & $\C^{\overline{2}n}\oplus\Lambda^{\overline{0}}\oplus\Lambda^{n-1}\oplus\{\C^{1}_{p,q}\Lambda^{n-2}\}$ \\ \hline
 $\check{\Q}^{\overline{0}}$ & $2$ & & $\C^{\overline{2}}\oplus\Lambda^{\overline{0}}\oplus\{\C^{1}_{p,q}\Lambda^{n-3}\}\oplus\{\C^{2}_{p,q}\Lambda^{n-4}\}$ \\ \hline
  $\Q^{\overline{0}}$ & $3$ & $\neq n$ & $\C^{\overline{2}n}\oplus\Lambda^{\overline{0}(1)}\oplus\{\C^{1}_{p,q}\Lambda^{n-3}\}\oplus\{\C^{2}_{p,q}\Lambda^{n-4}\}$ \\ \hline
 $\check{\Q}^{\overline{0}}$ & $1$ & & $\C^{\overline{2}}\oplus\Lambda^{\overline{0}}\oplus\C^{1}_{p,q}\Lambda^{n-2}\oplus\{\C^{2}_{p,q}\Lambda^{n-3}\}$ \\ \hline
 $\check{\Q}^{\overline{3}}$ & $2$ & & $\C^{\overline{2}}\oplus\Lambda^{\overline{0}(1)}\oplus\{\C^{1}_{p,q}\Lambda^{n-2}\}\oplus\{\C^{2}_{p,q}\Lambda^{n-3}\}$ \\ \hline
 $\Q^{\overline{0}}$ & $0$ & $\neq n$ & $\C^{\overline{2}n}\oplus\Lambda^{\overline{0}(1)}\oplus\{\C^{1}_{p,q}\Lambda^{n-2}\}\oplus\{\C^{2}_{p,q}\Lambda^{n-3}\}$\\ \hline
 $\check{\Q}^{\overline{3}}$ & $3$ & & $\C^{\overline{2}}\oplus\Lambda^{\overline{0}(1)}\oplus\{\C^{1}_{p,q}\Lambda^{n-1}\}\oplus\{\C^{2}_{p,q}\Lambda^{n-2}\}$ \\ \hline
$\Q^{\overline{3}}$ & $0$ & & $\C^{\overline{2}}\oplus\Lambda^{\overline{0}}\oplus\Lambda^{n-1}\oplus\{\C^{2}_{p,q}\Lambda^{n-2}\}$ \\
& & & $\oplus\{\C^{1}_{p,q}\Lambda^{n-2}\}\oplus\{\C^{1}_{p,q}\Lambda^{n-1}\}$ \\ \hline
$\Q^{\overline{3}}$ & $1$ & & $\C^{\overline{2}n}\oplus\Lambda^{\overline{0}}\oplus\Lambda^{n-2}\oplus\{\C^{2}_{p,q}\Lambda^{n-3}\}$ \\ 
& & & $\oplus\{\C^{1}_{p,q}\Lambda^{n-3}\}\oplus\{\C^{1}_{p,q}\Lambda^{n-2}\}$ \\ \hline
$\Q^{\overline{0}}$ & $2$ & & $\C^{\overline{2}}\oplus\Lambda^{\overline{0}(1)}\oplus\{\C^{1}_{p,q}\Lambda^{n-3}\}\oplus\{\C^{1}_{p,q}\Lambda^{n-2}\}$\\ 
& & & $\oplus\{\C^{2}_{p,q}\Lambda^{n-4}\}\oplus\{\C^{2}_{p,q}\Lambda^{n-3}\}$ \\ \hline
$\Q^{\overline{0}}$ & $1$ & $\neq n$ & $\C^{\overline{2}n}\oplus\Lambda^{\overline{0}(1)}\oplus\{\C^{1}_{p,q}\Lambda^{n-3}\}\oplus\{\C^{1}_{p,q}\Lambda^{n-2}\}$\\ 
& & & $\oplus\{\C^{2}_{p,q}\Lambda^{n-4}\}\oplus\{\C^{2}_{p,q}\Lambda^{n-3}\}$ \\ \hline
$\check{\Q}^{\overline{3}}$ & $0,1$ & & $\C^{\overline{2}}\oplus\Lambda^{\overline{0}(1)}\oplus\{\C^{1}_{p,q}\Lambda^{n-2}\}\oplus\{\C^{1}_{p,q}\Lambda^{n-1}\}$ \\
& & & $\oplus\{\C^{2}_{p,q}\Lambda^{n-3}\}\oplus\{\C^{2}_{p,q}\Lambda^{n-2}\}$ \\  \hline
\end{tabular}
\end{table}
\footnotetext{The first column of Table \ref{table_lie_alg} lists the Lie groups considered in the work. The fourth column provides the corresponding Lie algebra for each Lie group depending on $n\mod{4}$ and $r$ given in the second and third column respectively. The third column is empty for the Lie groups that have the same Lie algebra for any $r$. The sets in the fourth column are considered with respect to the commutator $[U,V]=UV-VU$.}
}

\begin{rem}
    Dimensions of the Lie algebras presented in Table \ref{table_lie_alg} can be calculated, using 
    \begin{eqnarray*}
        \!\!\!\!\!\!&\dim \C^{\overline{2}}=2^{n-2}-2^{\frac{n}{2}-1}\cos(\frac{\pi n}{4}),\quad\dim\C^{n}=1,
        \\
        \!\!\!\!\!\!&\dim \C^{k}_{p,q}=\frac{(p+q)!}{k!(p+q-k)!},\quad \dim\Lambda^{k}=
        \left\lbrace
            \begin{array}{lll}
        \frac{r!}{k!(r-k)!},&& \!\!\!\!\!r\geq k,
        \\
        0,&& \!\!\!\!\!r< k,
        \end{array}
            \right.
            \quad 0\leq k\leq n,
        \\
        \!\!\!\!\!\!&\dim\Lambda^{(1)}=
        \left\lbrace
            \begin{array}{lll}
            2^{r-1},&& \!\!\!\!\!r\neq0,
            \\
            0,&& \!\!\!\!\!r=0,
            \end{array}
            \right.
            \quad \dim\Lambda^{\overline{0}}=
        \left\lbrace
        \begin{array}{lll}
        2^{r-2}+2^{\frac{r}{2}-1}\cos(\frac{\pi r}{4}),&&\!\!\!\!\!r\neq0,
        \\
        1,&&\!\!\!\!\!r=0.
        \end{array}
            \right.
    \end{eqnarray*}
\end{rem}

Let us denote the Lie algebras of the Lie groups $\Q^{\overline{k}}$ and $\check{\Q}^{\overline{k}}$, $k=0,1,2,3$, by $\mathfrak{q}^{\overline{k}}=\mathfrak{q}^{\overline{k}}_{p,q,r}$ and $\check{\mathfrak{q}}^{\overline{k}}=\check{\mathfrak{q}}^{\overline{k}}_{p,q,r}$ respectively.

\begin{rem}
    In the case of the non-degenerate geometric algebra $\C_{p,q}$, we have for $k=0,1,2,3$, $m=0,2$, and $l=1,3$,
    \begin{eqnarray}
    \!\!\!\!\!\!\!\!\!\!&    \mathfrak{q}^{\overline{k}}_{p,q}=\check{\mathfrak{q}}^{\overline{2}}_{p,q}=\check{\mathfrak{q}}^{\overline{3}}_{p,q}=\C_{p,q},\quad \check{\mathfrak{q}}^{\overline{1}}_{p,q}=\check{\mathfrak{q}}^{\overline{0}}_{p,q}=\C^{0},\quad n=1;\label{lie1}
    \\
    \!\!\!\!\!\!\!\!\!\!& \mathfrak{q}^{\overline{3}}_{p,q}=\mathfrak{q}^{\overline{0}}_{p,q}=\check{\mathfrak{q}}^{\overline{2}}_{p,q}=\check{\mathfrak{q}}^{\overline{3}}_{p,q}=\C_{p,q},
    \\
    \!\!\!\!\!\!\!\!\!\!&\mathfrak{q}^{\overline{1}}_{p,q}=\mathfrak{q}^{\overline{2}}_{p,q}=\check{\mathfrak{q}}^{\overline{1}}_{p,q}=\check{\mathfrak{q}}^{\overline{0}}_{p,q}=\C^{(0)}_{p,q},\quad n=2;
    \\
    \!\!\!\!\!\!\!\!\!\!&  \mathfrak{q}^{\overline{3}}_{p,q}=\mathfrak{q}^{\overline{0}}_{p,q}=\C_{p,q},\;\;
    \mathfrak{q}^{\overline{1}}_{p,q}=\mathfrak{q}^{\overline{2}}_{p,q}=\C^{0\overline{2}3}_{p,q},\;\; \check{\mathfrak{q}}^{\overline{k}}_{p,q}=\C^{(0)}_{p,q},\;\; n=3;
    \\
    \!\!\!\!\!\!\!\!\!\!&   \mathfrak{q}^{\overline{m}}_{p,q}=\check{\mathfrak{q}}^{\overline{m}}_{p,q}=\C^{0\overline{2}n}_{p,q},\;\;   \mathfrak{q}^{\overline{l}}_{p,q}=\check{\mathfrak{q}}^{\overline{l}}_{p,q}=\C^{0\overline{2}}_{p,q},\;\; n=0\mod{4},\quad n\geq 4;
    \\
    \!\!\!\!\!\!\!\!\!\!& \mathfrak{q}^{\overline{k}}_{p,q}=\C^{0\overline{2}n}_{p,q},\quad  \check{\mathfrak{q}}^{\overline{k}}_{p,q}=\C^{0\overline{2}}_{p,q},\quad n=1\mod{2},\quad n\geq5;
    \\
    \!\!\!\!\!\!\!\!\!\!&   \mathfrak{q}^{\overline{k}}_{p,q}=\check{\mathfrak{q}}^{\overline{k}}_{p,q}=\C^{0\overline{2}}_{p,q},\quad n=2\mod{4},\quad n\geq 6.
    \end{eqnarray}
    Note that the obtained Lie algebras are considered in \cite{OnInner,GenSpin}.
    
\end{rem}

Let us consider the well-known Heisenberg Lie algebras $\mathfrak{heis}(3,\F)$ and $\mathfrak{heis}(4,\F)$ (see, for example, \cite{hall}):
\begin{eqnarray}
\mathfrak{heis}(3,\F)&=&\left\{
\begin{bmatrix}
0 & x_{12} & x_{13} \\
0 & 0 & x_{23} \\
0 & 0 & 0 \\
\end{bmatrix}\in\Mat(3,\F)\right\},\label{heis3}
\\
\mathfrak{heis}(4,\F)&=&\left\{
\begin{bmatrix}
0 & x_{12} & x_{13} & x_{14} \\
0 & 0 & 0 & x_{24} \\
0 & 0 & 0 & x_{34} \\
0 & 0 & 0 & 0 \\
\end{bmatrix}\in\Mat(4,\F)\right\}.\label{heis4}
\end{eqnarray}

\begin{rem}
    In the case of the Grassmann algebra $\C_{0,0,n}$, we have
    \begin{eqnarray}
    &\mathfrak{q}^{\overline{0}}_{0,0,n}=\mathfrak{q}^{\overline{2}}_{0,0,n}=\check{\mathfrak{q}}^{\overline{1}}_{0,0,n}=\check{\mathfrak{q}}^{\overline{3}}_{0,0,n}=\Lambda_n,\qquad \check{\mathfrak{q}}^{\overline{0}}_{0,0,n}=\Lambda^{(0)}_n,\qquad \forall n;\label{lie_1}
    \\
      &  \mathfrak{q}^{\overline{1}}_{0,0,n}=\check{\mathfrak{q}}^{\overline{2}}_{0,0,n}=\Lambda^{(0)n}_n,\qquad \mathfrak{q}^{\overline{3}}_{0,0,n}=\Lambda^{(0)n}_n\oplus\Lambda^{n-2}_n,\qquad \mbox{$n$ is odd};
      \\
      &\mathfrak{q}^{\overline{1}}_{0,0,n}=\Lambda^{(0)}_n,\qquad \check{\mathfrak{q}}^{\overline{2}}_{0,0,n}=\mathfrak{q}^{\overline{3}}_{0,0,n}=\Lambda^{(0)}_n\oplus\Lambda^{n-1}_n,\qquad \mbox{$n$ is even}.\label{lie_3}
    \end{eqnarray}
    Note that all the Lie algebras (\ref{lie_1})--(\ref{lie_3}) are solvable and nilpotent and the following Lie algebras are abelian:
    \begin{eqnarray*}
\!\!\!\!\!\!&\Lambda^{(0)}_n(\F)\cong  \bigoplus^{2^{n-1}}\mathfrak{gl}(1,\F),\quad \forall n;
\quad \Lambda^{(0)n}_n(\F)\cong  \bigoplus^{2^{n-1}+1}\mathfrak{gl}(1,\F),\quad \mbox{$n$ is odd};
\\
\!\!\!\!\!\!&\Lambda^{(0)n}_n(\F)\oplus\Lambda^{n-2}_n(\F)\cong \bigoplus^{2^{n-1}+\frac{n(n-1)}{2}}\mathfrak{gl}(1,\F),\quad \mbox{$n\neq3$ is odd};
\\
\!\!\!\!\!\!&\Lambda^{(0)}_n(\F)\oplus\Lambda^{n-1}_n(\F)\cong\bigoplus^{2^{n-1}+n}\mathfrak{gl}(1,\F),\;\;\mbox{$n\geq4$ is even}.
\end{eqnarray*}
Also note that
\begin{eqnarray}\label{formula_lie_alg1}
\Lambda_2(\F)\cong \mathfrak{gl}(1,\F)\oplus\mathfrak{heis}(3,\F).
\end{eqnarray}
\end{rem}

\begin{ex}
We have the following isomorphisms:
\begin{eqnarray}
    &\check{\mathfrak{q}}^{\overline{0}}_{1,0,2}(\F)\cong \mathfrak{gl}(1,\F)\oplus\mathfrak{heis}(3,\F),\quad     \check{\mathfrak{q}}^{\overline{0}}_{1,1,1}(\BR)\cong\mathfrak{gl}(1,\BR)\oplus\mathfrak{p}(1,1),\label{ex_1}
    \\
    &\check{\mathfrak{q}}^{\overline{0}}_{3,0,0}(\BR)\cong \mathfrak{gl}(1,\BR)\oplus\mathfrak{su}(2),\quad \check{\mathfrak{q}}^{\overline{0}}_{3,0,0}(\BC)\cong \mathfrak{gl}(1,\BC)\oplus\mathfrak{sl}(2,\BC),\label{ex_1_1}
    \\
    &\check{\mathfrak{q}}^{\overline{1}}_{1,0,2}(\BF)\cong\mathfrak{gl}(1,\F)\oplus\mathfrak{heis}(4,\F),
\end{eqnarray}
and, by adding abelian Lie subalgebras, we get 
\begin{eqnarray}
\!\!\!\!\!\!\!\!\!\!\!\!\!\!\!&{\mathfrak{q}}^{\overline{1}}_{1,0,2}(\F)\cong \bigoplus^2\mathfrak{gl}(1,\F)\oplus\mathfrak{heis}(3,\F),
\quad{\mathfrak{q}}^{\overline{1}}_{1,1,1}(\BR)\cong\bigoplus^2\mathfrak{gl}(1,\BR)\oplus\mathfrak{p}(1,1),
\\
\!\!\!\!\!\!\!\!\!\!\!\!\!\!\!&{\mathfrak{q}}^{\overline{1}}_{3,0,0}(\BR)\cong\bigoplus^2\mathfrak{gl}(1,\BR)\oplus\mathfrak{su}(2),
\quad{\mathfrak{q}}^{\overline{1}}_{3,0,0}(\BC)\cong\bigoplus^2\mathfrak{gl}(1,\BC)\oplus\mathfrak{sl}(2,\BC),
\\
\!\!\!\!\!\!\!\!\!\!\!\!\!\!\!&\check{\mathfrak{q}}^{\overline{3}}_{1,0,2}(\BF)\cong \mathfrak{gl}(1,\F)\oplus\mathfrak{gl}(1,\F)\oplus\mathfrak{heis}(4,\F),\label{ex_f}
\end{eqnarray}
where $\mathfrak{heis}(3,\F)$ (\ref{heis3}) and $\mathfrak{heis}(4,\F)$ (\ref{heis4}) are the Heisenberg Lie algebras, $\mathfrak{p}(1,1)$ is the three-dimensional Poincaré Lie algebra (see, for example, \cite{poi11}), which is defined by the commutation relations of the basis elements $x_1$, $x_2$, and $x_3$:
\begin{eqnarray}\label{poin_comm}
    [x_1,x_3]=x_1,\quad [x_2,x_3]=-x_2,\quad [x_1,x_2]=0,
\end{eqnarray}
and 
\begin{eqnarray}
\mathfrak{sl}(n,\F)&=&\{A\in\Mat(n,\F):\quad \tr A=0\},\label{sl}
\\
\mathfrak{su}(n)&=&\{A\in\Mat(n,\BC):\quad A^{\dagger}=-A,\quad \tr A=0\}.\label{su}
\end{eqnarray}
Note that we can get other isomorphisms of Lie algebras, using Table \ref{table_lie_alg}.
\end{ex}

\section{Conclusions}\label{section_conclusions}

In this work, we consider the families of Lie groups $\Gamma^{\overline{k}}$ (\ref{gamma_ov_k}), $\check{\Gamma}^{\overline{k}}$ (\ref{gamma_ov_chk}), and $\tilde{\Gamma}^{\overline{k}}$ (\ref{gamma_ov_tik}), preserving the subspaces $\C^{\overline{k}}$, $k=0,1,2,3$, determined by the grade involution and reversion under the adjoint representation $\ad$ and the twisted adjoint representations $\check{\ad}$ and $\tilde{\ad}$ respectively in degenerate and non-degenerate geometric algebras $\C$. 
We introduce the families of Lie groups $\Q^{\overline{k}}$ (\ref{def_Q1})--(\ref{def_Q0}) and $\check{\Q}^{\overline{k}}$ (\ref{def_chQ1})--(\ref{def_chQ0}), $k=0,1,2,3$,  defined using the norm functions $\psi$ and $\chi$ (\ref{norm_functions}). In the definitions of these Lie groups, we also use centralizers $\Z^{m}$ (\ref{def_Zm}), twisted centralizers $\check{\Z}^{m}$ (\ref{def_chZm}), and their properties.
In Theorems \ref{maintheo_q} and \ref{maintheo_checkq}, we prove that the groups $\Q^{\overline{k}}$ and $\check{\Q}^{\overline{k}}$ are related to the groups 
$\Gamma^{\overline{k}}$, $\check{\Gamma}^{\overline{k}}$, and $\tilde{\Gamma}^{\overline{k}}$
in the following way\footnote{We also wish to include an important comment from an anonymous reviewer concerning these equalities: 'Now we have the explicit conditions (which may be written as quadratic polynomials in coordinates) for these generalized Clifford and Lipschitz groups. This helps enable their use in applications and could support future investigation into their properties as algebraic groups'.}:
\begin{eqnarray}
    \Gamma^{\overline{k}}=\Q^{\overline{k}},\qquad \check{\Gamma}^{\overline{k}}=\tilde{\Gamma}^{\overline{k}}=\check{\Q}^{\overline{k}},\qquad k=1,3;
    \\
    \Gamma^{\overline{k}}=\tilde{\Gamma}^{\overline{k}}=\Q^{\overline{k}},\qquad \check{\Gamma}^{\overline{k}}=\check{\Q}^{\overline{k}},\qquad k=0,2.
\end{eqnarray}
In Table \ref{table_all_groups}, we present a comprehensive list of the groups under consideration in this work and their defining conditions. Namely, the first column indicates the group, while the second column contains the set to which the norm functions $\psi$ and $\chi$  of the group's elements belong  according to the group's definition.

\begin{table}[h]
\caption{The considered Lie groups}\label{table_all_groups} 
\begin{tabular}{l | l}  \hline
Lie group & $\widetilde{T}T$, $\stackrel{}{\widehat{\widetilde{T}}T}
$ \\ \hline
$\Q^{\overline{1}}=\Gamma^{\overline{1}}$ & $\Z^{1\times}=\ker(\ad)$ \\ 
$\Q^{\overline{2}}=\Gamma^{\overline{2}}=\tilde{\Gamma}^{\overline{2}}$ & $\Z^{2\times}$\\ 
$\Q^{\overline{3}}=\Gamma^{\overline{3}}$ & $\Z^{3\times}$\\ 
$\Q^{\overline{0}}=\Gamma^{\overline{0}}=\tilde{\Gamma}^{\overline{0}}$ & $\Z^{4\times}$\\ 
$\check{\Q}^{\overline{1}}=\check{\Gamma}^{\overline{1}}=\tilde{\Gamma}^{\overline{1}}$ & $\check{\Z}^{1\times}=\ker(\tilde{\ad})$ \\ 
$\check{\Q}^{\overline{2}}=\check{\Gamma}^{\overline{2}}$ & $\check{\Z}^{2\times}$\\ 
$\check{\Q}^{\overline{3}}=\check{\Gamma}^{\overline{3}}=\tilde{\Gamma}^{\overline{3}}$ & $\check{\Z}^{3\times}$\\
$\check{\Q}^{\overline{0}}=\check{\Gamma}^{\overline{0}}$ & $(\Z^4\cap\C^{(0)})^{\times}$  \\ \hline
\end{tabular}
\end{table}

In the special case of small dimensions, the Lie groups considered in this paper are related to the well-known matrix groups. Several examples are provided in Section \ref{section_examples}.
For instance, in the case of the Grassmann algebras $\Lambda_1$ and $\Lambda_2$, all the groups in this paper can be realized as subgroups of the upper triangular matrix groups $\UT(2,\F)$ and $\UT(4,\F)$. 
Also in the case of $\Lambda_2$, the higher-dimensional Heisenberg group $\Heis_4$ is closely related to some of the considered groups.
Namely, the group $\Q^{\overline{0}}_{0,0,2}=\Q^{\overline{2}}_{0,0,2}=\check{\Q}^{\overline{1}}_{0,0,2}=\check{\Q}^{\overline{2}}_{0,0,2}$ is isomorphic to a matrix group, and its unipotent subgroup is a subgroup of $\Heis_4$.

In Theorem \ref{thm_alg} and Table \ref{table_lie_alg}, we study the families of Lie algebras corresponding to the considered families of Lie groups. 
The examples of the connection between the corresponding Lie algebras and the classical Lie algebras are provided in Section \ref{section_liealg}.
For instance, in the case of $\C_{1,0,2}$, $\C_{1,1,1}$, and $\C_{3,0,0}$, some of the considered Lie algebras are isomorphic to the direct sums of the general linear Lie algebras, Heisenberg Lie algebras $\mathfrak{heis}(3,\F)$ and $\mathfrak{heis}(4,\F)$, special unitary Lie algebra $\mathfrak{su}(2)$, special linear Lie algebra $\mathfrak{sl}(2,\BC)$, and Poincaré Lie algebra $\mathfrak{p}(1,1)$ (see (\ref{ex_1})--(\ref{ex_f})).
The relation between the considered Lie groups and Lie algebras and the well-known Lie groups and Lie algebras in the case of arbitrary dimension and signature of $\C$ requires further research.
In the special case of the non-degenerate geometric algebra $\C_{p,q}$, this issue is addressed in \cite{lg1_,lg2_,lg3_}.

The groups $\Gamma^{\overline{k}}$, $\check{\Gamma}^{\overline{k}}$, and $\tilde{\Gamma}^{\overline{k}}$ can be viewed as generalizations of the Clifford groups $\Gamma$ (\ref{def_lg}) and the Lipschitz group $\Gamma^{\pm}$ (\ref{def_lg}). Some of them coincide with the Clifford and Lipschitz groups in the cases $n\leq4$ (see Remarks \ref{rem_cg} and \ref{rem_lg}). The considered groups are important for the theory of spin groups. 
By analogy with the ordinary spin groups, the groups $\Gamma^{\overline{k}}$, $\check{\Gamma}^{\overline{k}}$, and $\tilde{\Gamma}^{\overline{k}}$ can be normalized using the norm functions $\psi$ and $\chi$ (\ref{norm_functions}). 
These normalized subgroups can be viewed as generalized degenerate spin groups and can be important for applications. The study of these groups is the direction for the further research.
Moreover, the generalized Lipschitz groups $\check{\Gamma}^{\overline{k}}$ and $\tilde{\Gamma}^{\overline{k}}$ are interesting for applications in construction of neural networks equivariant with respect to pseudo-orthogonal transformations (rotations, reflections, etc.). 
Specifically, since ordinary Lipschitz groups $\Gamma^{\pm}$ are subgroups of the generalized Lipschitz groups, equivariance of a mapping with respect to these groups implies its pseudo-orthogonal groups equivariance. 
Therefore, parametrization of any mappings equivariant with respect to the generalized Lipschitz groups can be employed
in construction of new equivariant neural networks architectures \cite{glgenn0,glgenn}.
So, the families of Lie groups studied in this paper can be useful for applications of geometric algebras in physics \cite{br1,br2,phys,hestenes,ce}, quantum mechanics \cite{br2}, computer science, in particular, for neural networks and machine learning \cite{b1,tf,cNN0,cNN,glgenn0,glgenn,zhdanov}, image processing \cite{bayro_new1,hild_new,ce}, and in other areas.

\section*{Acknowledgements}

The results of this paper were reported at The 9th conference on Applied Geometric Algebras in Computer Science and Engineering, Amsterdam, Netherlands, August 2024. The authors are grateful to the organizers and the participants of this conference for fruitful discussions. The authors are grateful to the reviewers for their careful reading of the paper and helpful comments on how to improve the presentation.

This work is supported by the Russian Science Foundation (project 23-71-10028), https://rscf.ru/en/project/23-71-10028/.

\medskip

\noindent{\bf Data availability} Data sharing not applicable to this article as no datasets were generated or analyzed during the current study.

\medskip

\noindent{\bf Declarations}\\
\noindent{\bf Conflict of interest} The authors declare that they have no conflict of interest.

\appendix

\section{Summary of notation}\label{section_notation}

For the readers' convenience, we summarize all the notation used throughout the paper in Tables \ref{table_notation1}--\ref{table_notation2}. The tables list each notation, its meaning, and the place where it appears for the first time.

\begin{table}[h]
\caption{Summary of notation (part 1)}\label{table_notation1} 
\begin{tabular}{| c | c | c |}  \hline
Notation  & Meaning & First mention \\ \hline
$\C=\C_{p,q,r}$ & \parbox{5.2cm}{\begin{center}(Clifford) geometric algebra over the real $\BR^{p,q,r}$ or complex $\BC^{p+q,0,r}$ vector space\end{center}} & page \pageref{section_Cpqr}\\ \cline{2-2}
$\F$ & \parbox{5.2cm}{\begin{center}field of real or complex numbers in the cases $\BR^{p,q,r}$ and $\BC^{p+q,0,r}$  respectively\end{center}} & page \pageref{section_Cpqr}  \\ \cline{2-2}
$\Lambda=\Lambda_r$ & Grassmann subalgebra $\C_{0,0,r}$ & page \pageref{section_Cpqr} \\ \cline{2-2}
$e_1,\ldots,e_n$ & generators of $\C_{p,q,r}$ & page \pageref{section_Cpqr} \\ \cline{2-2}
$\C^{0}$ & subspace of grade $0$ & page \pageref{def_Cgeqk}\\ \cline{2-2}
$\C^k=\C^{k}$ & subspace of fixed  grade $k=1,\ldots,n$ & page \pageref{def_Cgeqk} \\ \cline{2-2}
$\C^{\geq k}$, $\C^{\leq k}$ & \parbox{5.2cm}{\begin{center}direct sums of subspaces of grades larger/smaller or equal to $k$\end{center}} & \parbox{2.2cm}{\begin{center}formula \eqref{def_Cgeqk}\end{center}} \\ \cline{2-2}
$\{\C^k_{p,q}\Lambd^l\}$ & \parbox{5.2cm}{\begin{center}subspace of $\C_{p,q,r}$ spanned by the elements of the form $ab$, where $a\in \C^k_{p,q}$ and $b\in\Lambd^l_r$\end{center}} & page \pageref{def_Cgeqk}\\ \hline
$\stackrel{}{\widehat{U}}$ & grade involute of $U\in\C$ & page \pageref{def_even_odd} \\ \cline{2-2}
$\widetilde{U}$ & reverse of $U\in\C$ & page \pageref{def_even_odd} \\ \cline{2-2}
$\widehat{\widetilde{U}}$ & Clifford conjugate of $U\in\C$ & page \pageref{def_even_odd} \\ \hline
$\C^{(0)}$, $\C^{(1)}$ & even and odd subspaces & formula \eqref{def_even_odd} \\ \cline{2-2}
$\langle\rangle_{(0)}$, $\langle\rangle_{(1)}$ & \parbox{5.2cm}{\begin{center}projections onto $\C^{(0)}$, $\C^{(1)}$\end{center}} & \parbox{2.2cm}{\begin{center}formulas \eqref{proj1}--\eqref{H_even}\end{center}} \\ \hline
$\C^{\overline{k}}$ & \parbox{5.2cm}{\begin{center}subspaces determined by grade involution and reversion\end{center}} & formula \eqref{qtdef}\\ \cline{2-2}
$\C^{\overline{kl}}$ & \parbox{5.2cm}{\begin{center}direct sums of $\C^{\overline{k}}$, $k=0,1,2,3$\end{center}} & page \pageref{qtdef} \\ \hline
$\H^\times$ & \parbox{5.2cm}{\begin{center}subset of all invertible elements of a set $\H$\end{center}}  & page \pageref{def_cg} \\ \hline
$\Gamma$ & Clifford group & formula \eqref{def_cg}\\ \cline{2-2}
$\Gamma^{\pm}$ & Lipschitz group & formula \eqref{def_lg}\\ \hline
$\psi$, $\chi$ & norm functions of $\C$ elements & formula \eqref{norm_functions} \\ \hline
$\ad$ & \parbox{5.2cm}{\begin{center}adjoint representation ${\ad}_{T}(U)=TU T^{-1}$\end{center}} &  formula (\ref{ar})\\ \cline{2-2}
$\check{\ad}$ & \parbox{5.2cm}{\begin{center}twisted adjoint representation $\check{\ad}_{T}(U)=\widehat{T}U T^{-1}$\end{center}} & \parbox{2.2cm}{\begin{center}formula (\ref{twa1})  \end{center}} \\ \cline{2-2}
$\tilde{\ad}$ & \parbox{5.2cm}{\begin{center}twisted adjoint representation $\tilde{\ad}_{T}(U)=TU_0 T^{-1}+\widehat{T} U_1 T^{-1}$\end{center}} & \parbox{2.2cm}{\begin{center}formula (\ref{twa22})\end{center}} \\ \cline{2-2}
\parbox{2.2cm}{\begin{center}$\ker(\ad)$, $\ker{(\check{\ad})}$, $\ker{(\tilde{\ad})}$\end{center}} & kernels of $\ad$, $\check{\ad}$, and $\tilde{\ad}$ & page \pageref{ker_ad}\\ \hline
$\Z$ & center of $\C$ & formula \eqref{Zpqr} \\  \hline
\end{tabular}
\end{table}

\begin{table}[h]
\caption{Summary of notation (part 2)}\label{table_notation2} 
\begin{tabular}{| c | c | c |}  \hline
Notation  & Meaning & First mention \\ \hline
$\Z^m$, $\check{\Z}^m$ & \parbox{5.2cm}{\begin{center}centralizers and twisted centralizers respectively of $\C^{m}$, $m=0,\ldots,n$\end{center}} & \parbox{2.2cm}{\begin{center}formulas \eqref{def_Zm}--\eqref{def_chZm}\end{center}} \\ \cline{2-2}
$\Z^{\overline{k}}$, $\check{\Z}^{\overline{k}} $& \parbox{5.2cm}{\begin{center}centralizers and twisted centralizers respectively of $\C^{\overline{k}}$, $k=0,1,2,3$\end{center}} & \parbox{2.2cm}{\begin{center}formulas \eqref{def_CC_ov}--\eqref{def_chCC_ov}\end{center}}\\ 
\hline
\parbox{3.2cm}{\begin{center}$\Gamm^{(l)}$,
$\check{\Gamm}^{(l)}$, $\tilde{\Gamm}^{(l)}$\end{center}} & \parbox{5.2cm}{\begin{center}groups preserving $\C^{(l)}$, $l=0,1$, under $\ad$, $\check{\ad}$, and $\tilde{\ad}$ \end{center}} & \parbox{2.2cm}{\begin{center}formulas \eqref{p_f1}--\eqref{p_f}\end{center}}\\ \cline{2-2}
\parbox{2.2cm}{\begin{center}$\P$, $\P^{\pm\Lambda}$ \end{center}}  & \parbox{4.2cm}{\begin{center}groups preserving $\C^{(1)}$ under $\ad$, $\check{\ad}$, and $\tilde{\ad}$ \end{center}} & \parbox{2.2cm}{\begin{center}formulas \eqref{prev_gr}--\eqref{prev_gr_2}\end{center}}\\ \cline{2-2}
\parbox{2.2cm}{\begin{center}$\P^{\pm}$, $\P^{\Lambda}$ \end{center}} & \parbox{4.2cm}{\begin{center}groups preserving $\C^{(0)}$ under $\ad$, $\check{\ad}$, and $\tilde{\ad}$ \end{center}} & \parbox{2.2cm}{\begin{center}formulas \eqref{prev_gr}--\eqref{prev_gr_2}\end{center}}\\ \hline
\parbox{3.2cm}{\begin{center}$\Gamm^{\overline{kl}}$,
$\check{\Gamm}^{\overline{kl}}$\end{center}} & \parbox{5.2cm}{\begin{center}groups preserving $\C^{\overline{kl}}$ under $\ad$ and $\check{\ad}$\end{center}} & \parbox{2.2cm}{\begin{center}formulas \eqref{def_Gamma_kl}--\eqref{def_chGamma_kl}\end{center}}\\ \cline{2-2}
\parbox{2.2cm}{\begin{center}$\A^{\overline{01}}$, $\B^{\overline{12}}$, $\A^{\overline{23}}$, $\B^{\overline{03}}$\end{center}}  & \parbox{4.2cm}{\begin{center}groups preserving $\C^{\overline{kl}}$ under $\ad$\end{center}} & \parbox{2.2cm}{\begin{center}formulas \eqref{def_A01}--\eqref{def_B03}\end{center}}\\ \cline{2-2}
\parbox{2.2cm}{\begin{center}$\check{\A}^{\overline{12}}$, $\check{\B}^{\overline{01}}$, $\check{\A}^{\overline{03}}$, $\check{\B}^{\overline{23}}$ \end{center}} & \parbox{4.2cm}{\begin{center}groups preserving $\C^{\overline{kl}}$ under $\check{\ad}$\end{center}} & \parbox{2.2cm}{\begin{center}formulas \eqref{def_chA12}--\eqref{def_chB23}\end{center}}\\ \hline
\parbox{3.2cm}{\begin{center}$\Gamm^{\overline{k}}$,
$\check{\Gamm}^{\overline{k}}$, $\tilde{\Gamma}^{\overline{k}}$ \end{center}} & \parbox{5.2cm}{\begin{center}groups preserving $\C^{\overline{k}}$ under $\ad$, $\check{\ad}$, and $\tilde{\ad}$ \end{center}} & \parbox{2.2cm}{\begin{center}formulas \eqref{gamma_ov_k}--\eqref{gamma_ov_tik}\end{center}}\\ \cline{2-2}
\parbox{2.2cm}{\begin{center}$\Q^{\overline{k}}$\end{center}}  & \parbox{4.2cm}{\begin{center}groups preserving $\C^{\overline{k}}$ under $\ad$ (if $k=0,1,2,3$) and under $\tilde{\ad}$ (if $k=0,2$) \end{center}} & \parbox{2.2cm}{\begin{center}formulas \eqref{def_Q1}--\eqref{def_Q0}\end{center}}\\ \cline{2-2}
\parbox{2.2cm}{\begin{center}$\check{\Q}^{\overline{k}}$ \end{center}} & \parbox{4.2cm}{\begin{center}groups preserving $\C^{\overline{k}}$ under $\check{\ad}$ (if $k=0,1,2,3$) and under $\tilde{\ad}$ (if $k=1,3$) \end{center}} & \parbox{2.2cm}{\begin{center}formulas \eqref{def_chQ1}--\eqref{def_chQ0}\end{center}}\\ \hline
$\Mat(n,\F)$ & algebra of matrices of size $n$ & page \pageref{heis}\\ \cline{2-2}
$\UT(n,\F)$ & \parbox{4.2cm}{\begin{center}group of upper triangular matrices\end{center}} & \parbox{2.2cm}{\begin{center}formulas \eqref{ut_2}--\eqref{ut_4}\end{center}} \\ \cline{2-2}
$\SUT(n,\F)$ & unipotent subgroup of $\UT(n,\F)$ & formula \eqref{sut}\\ \cline{2-2}
$\Heis_n$ & \parbox{4.2cm}{\begin{center}higher-dimensional Heisenberg group\end{center}} & formula \eqref{heis}\\ \hline
$\mathfrak{q}^{\overline{k}}$, $\check{\mathfrak{q}}^{\overline{k}}$ & \parbox{4.2cm}{\begin{center}Lie algebras of the Lie groups $\Q^{\overline{k}}$ and $\check{\Q}^{\overline{k}}$\end{center}} & page \pageref{lie1} \\ \cline{2-2}
$\mathfrak{gl}(1,\F)$ & \parbox{4.2cm}{\begin{center}general linear Lie algebra\end{center}} & page \pageref{formula_lie_alg1} \\ \cline{2-2}
$\mathfrak{heis}(3,\F)$, $\mathfrak{heis}(4,\F)$ & Heisenberg Lie algebras & \parbox{2.2cm}{\begin{center}formulas \eqref{heis3}--\eqref{heis4}\end{center}} \\ \cline{2-2}
$\mathfrak{p}(1,1)$ & Poincaré Lie algebra & formula \eqref{ex_1} \\ \cline{2-2}
$\mathfrak{su}(2)$ & special unitary Lie algebra & formula \eqref{ex_1_1}\\ \cline{2-2}
$\mathfrak{sl}(2,\BC)$ & special linear Lie algebra & formula \eqref{ex_1_1} \\ \hline
\end{tabular}
\end{table}

\bibliographystyle{spmpsci}

\end{document}